\documentclass[12pt,oneside]{amsart}
\usepackage[latin9]{inputenc}
\usepackage{mathtools}
\usepackage{amsthm}
\usepackage{amsbsy}
\usepackage{amstext}
\usepackage{amssymb}
\PassOptionsToPackage{normalem}{ulem}
\usepackage{ulem}

\makeatletter

\newcommand{\lyxmathsym}[1]{\ifmmode\begingroup\def\b@ld{bold}
  \text{\ifx\math@version\b@ld\bfseries\fi#1}\endgroup\else#1\fi}

\providecommand{\tabularnewline}{\\}

\numberwithin{equation}{section}
\numberwithin{figure}{section}
 \theoremstyle{definition}
 \newtheorem*{defn*}{\protect\definitionname}
\newenvironment{lyxlist}[1]
{\begin{list}{}
{\settowidth{\labelwidth}{#1}
 \setlength{\leftmargin}{\labelwidth}
 \addtolength{\leftmargin}{\labelsep}
 }}
{\end{list}}
  \theoremstyle{plain}
  \newtheorem*{fact*}{\protect\factname}
\theoremstyle{plain}
\newtheorem{thm}{\protect\theoremname}[section]
  \theoremstyle{definition}
  \newtheorem{defn}[thm]{\protect\definitionname}
  \theoremstyle{remark}
  \newtheorem*{rem*}{\protect\remarkname}
  \theoremstyle{plain}
  \newtheorem{lem}[thm]{\protect\lemmaname}
  \theoremstyle{plain}
  \newtheorem{cor}[thm]{\protect\corollaryname}
  \theoremstyle{remark}
  \newtheorem{rem}[thm]{\protect\remarkname}

\@ifundefined{date}{}{\date{}}

\makeatother

  \providecommand{\corollaryname}{Corollary}
  \providecommand{\definitionname}{Definition}
  \providecommand{\factname}{Fact}
  \providecommand{\lemmaname}{Lemma}
  \providecommand{\remarkname}{Remark}
\providecommand{\theoremname}{Theorem}

\begin{document}

\title[Equality of seven fundamental sets connected with $A(K)$]{Equality of seven fundamental sets connected with $A(K)$: analytic
capacity-free proofs}

\author{John M. Bachar Jr.}

\address{Department of Mathematics, California State University, Long Beach,
1250 Bellflower Blvd, Long Beach, CA 90840, USA}

\email{jmbachar@sbcglobal.net}



\date{Jan 2015}
\begin{abstract}
Let $K$ be any compact set in $\mathbb{C}$ with connected complement,
let $A(K)$ be the uniform algebra of all complex functions continuous
on $K$ and holomorphic on $\mathrm{int}(K)$, let $\partial K$ be
the topological boundary of $K$, let $z\in K$, and let $M_{z}$
be the maximal ideal of functions in $A(K)$ vanishing at $z$. Using
only facts from classical complex analytic function theory, and without
using any results from the theory of analytic capacity, we prove the
following: $z\in\partial K$ iff $z$ is a peak point for $A(K)$
iff $z$ belongs to the Shilov boundary of $A(K)$ iff $z$ belongs
to the Bishop minimal boundary of $A(K)$ iff $M_{z}$ has a bounded
approximate identity iff $z$ satisfies the Bishop $\frac{1}{4}-\frac{3}{4}$
property iff $z$ is a strong boundary point for $A\left(K\right)$.
More specifically, the only results used in all proofs come from classical
analytic function theory, properties of open connected sets in the
complex plane, the Caratherodory extension theorem, the Riemann mapping
theorem, Euler's formula, Rudin's estimates on finite complex products,
properties of linear fractional transformations, the $\alpha$-th
root function in the complex plane, some new and striking Jordan curve
constructions (Jordan kissing paths), and the Cohen Factorization
theorem. No results are used from the theory of analytic capacity,
the theory of representing or annihilating measures, Dirichlet algebra
theory, Choquet boundary theory, or the Walsh-Lebesgue theorem.
\end{abstract}

\maketitle

\section{Introduction}

\global\long\def\Ric{\mathrm{Ric}}
\global\long\def\Hess{\mathrm{\mathrm{Hess}}}
\global\long\def\scal{\mathrm{\mathrm{scal}}}
\global\long\def\bbR{\mathbb{\mathbb{R}}}
\global\long\def\span{\mathrm{span}}
\global\long\def\rank{\mathrm{rank}}
\global\long\def\bbC{\mathbb{\mathbb{C}}}
\global\long\def\bbF{\mathbb{\mathbb{F}}}
\global\long\def\bbP{\mathbb{\mathbb{P}}}
\global\long\def\Aut{\mathrm{Aut}}
\global\long\def\End{\mathrm{End}}
\global\long\def\Hom{\mathrm{Hom}}
\global\long\def\abs#1{\left|#1\right|}
\global\long\def\norm#1{\left\Vert #1\right\Vert }
\global\long\def\div{\mathrm{div}}
\global\long\def\ad{\mathrm{ad}}
\global\long\def\tr{\mathrm{tr}}
\global\long\def\bfq{{\bf q}}
\global\long\def\bfv{{\bf v}}
\global\long\def\bfa{{\bf a}}
\global\long\def\bfj{{\bf j}}
\global\long\def\bfN{{\bf N}}
\global\long\def\bfT{{\bf T}}
\global\long\def\bfB{{\bf B}}
\global\long\def\bft{{\bf t}}
\global\long\def\bfn{{\bf n}}
\global\long\def\rmI{\textrm{I}}
\global\long\def\rmII{\textrm{II}}
\global\long\def\bfS{{\bf S}}
\global\long\def\sm{\setminus}
\global\long\def\int{\mathrm{int}}
\global\long\def\arc{\mathrm{arc}}

In \cite{Bachar1}, the following was established. Let $K\subset\mathbb{C}$
be compact, let $P(K)$ be the uniform closure in $C(K)$ of all polynomial
functions restricted to $K$, and let $A(K)$ be the uniform algebra
of all functions continuous on $K$ and holomorphic on $\int(K)$.
It is known (Mergelyan) that $P(K)=A(K)$ if and only if $\mathbb{C}\setminus K$
is connected. (The implication $P(K)=A(K)$ implies $\mathbb{C}\backslash K$
is connected is proved in \cite{Stout} without using results from
the theory of analytic capacity; for the converse implication, Mergelyan's
original proof depended on results from the theory of analytic capacity,
but the proofs given in\cite{Stout} (Stout), \cite{Browder} (F.
Browder), \cite{Rudin} (Rudin) and \cite{Carleson} (Carleson) do
not use analytic capacity theory). Let $\mathcal{K}=$\{$K$$\subset\mathbb{C}$:
$K$ is compact and $\mathbb{C}\setminus K$ is connected\}, let $\partial K$
be the topological boundary of $K$ and let $\mathcal{P}(P(K))=$
\{$z\in K$: there exists $f\in P(K)$ such that $f(z)=1$, $|f(w)|<1$
for $w\in K\setminus\left\{ z\right\} $\}, i.e., the set of peak
points relative to $P(K)$. By the maximum modulus principle, $\mathcal{P}(P(K))\subseteq\partial K$.
A long standing problem is to determine whether $\mathcal{P}(P(K))=\partial K$
for all $K\in\mathcal{K}$. The rest of this paper contains the details
of such a proof.
\begin{defn*}
In what follows, various types of topological boundary points are
used:
\begin{lyxlist}{00.00.0000}
\item [{(1)}] $(\partial K)_{\text{I}}=$ \{$z_{0}\in\partial K$: every
open ngbd of $z_{0}$ intersects $\int(K)$\} (denoted type I);
\item [{(2)}] $((\partial K)_{\text{II}}=$ \{$z_{0}\in\partial K$: some
open ngbd of $z_{0}$ does not intersect $\int(K)$\} (denoted type
II); 
\item [{(3)}] $z\in\partial K$ is \emph{circularly accessible (ca)} iff
there is a $w\in\mathbb{C}\setminus K$, an $r>0$, and a closed disk
$\overline{D}(w;r)=\left\{ u\in\mathbb{C}:\big|u-w\big|\leq r\right\} $
such that $\overline{D}(w;r)\cap K=\left\{ z\right\} $ and $\overline{D}(w;r)\sm\left\{ z\right\} \subset\bbC\sm K$;
\item [{(4)}] $z\in\partial K$ is \emph{segmentally accessible (sa)} iff
there is a segment $\left[z,w\right]=\left\{ u\in\mathbb{C}:u=z+t(w-z),0\leq t\leq1\right\} $
such that $\left[z,w\right]\setminus\left\{ z\right\} \subset\mathbb{C}\setminus K$;
\item [{(5)}] $z\in\partial K$ is a \emph{Jordan escape point (Je)} iff
there exists a Jordan curve, $\Gamma_{z}$, initiating and terminating
at $z$ such that $K\setminus\left\{ z\right\} $ is contained in
the bounded component of $\Gamma_{z}$ (which is simply-connected
and connected) and such that the winding number of $\Gamma_{z}$ w.r.t.
$K\setminus\left\{ z\right\} $ is +1;
\item [{(6)}] $z\in\partial K$ is an \emph{escape point (ep)} iff there
is a continuous one-to-one function $f:\left[0,1\right]\mapsto\mathbb{C}$
such that $f(0)=z$ and $f((0,1])\subset\mathbb{C}\setminus K$;
\item [{(7)}] $\mathcal{P}_{\text{ca}}(K)$, $\mathcal{P}_{\text{sa}}(K)$,
$\mathcal{P}_{\text{Je}}(K)$, $\mathcal{P}_{\text{ep}}(K)$ denote
the sets of ca, sa, Je and ep points of $\partial K$, respectively.
\end{lyxlist}
\end{defn*}
\medskip{}

\begin{fact*}
In \cite{Bachar1}, the following facts are proved;
\begin{lyxlist}{00.00.0000}
\item [{\emph{(i)}}] $(\partial K)_{\text{II}}\subseteq\mathcal{P}(P(K))$
(\cite[THEOREM 5.1.ii]{Bachar1}).
\item [{\emph{(ii)}}] If $\int(K)=\emptyset$, then $K=(\partial K)_{\text{II}}=\mathcal{P}(P(K))$;
\item [{\emph{(iii)}}] $\mathcal{P}_{\textrm{ca}}(K)\subseteq\mathcal{P}_{\textrm{sa}}(K)\subseteq\mathcal{P}_{\text{Je}}(K)=\mathcal{P}_{\text{ep}}(K)\subseteq\mathcal{P}(P(K))\subseteq\partial K$
and $\mathcal{P}_{\textrm{ca}}(K)$ is dense in $\partial K$.
\item [{\emph{(iv)}}] If $\int(K)\neq\emptyset$ and if $\mathcal{P}_{\textrm{ca}}(K)=(\partial K)_{\text{I}}$,
or $\mathcal{P}_{\text{sa}}(K)=(\partial K)_{\text{I}}$, or $\mathcal{P}_{\text{Je}}(K)=(\partial K)_{\text{I}}$,
or $\mathcal{P}_{\text{ep}}(K)=(\partial K)_{\text{I}}$, then $\mathcal{P}(P(K))=\partial K$,
i.e., every boundary point is a peak point.
\item [{\emph{(v)}}] If $\int(K)\neq\emptyset$ and if $\partial(\int(K))$
is Jordan curve, then every point of $\partial K$ is a peak point
for $P(K)$ (\cite[THEOREM 5.5]{Bachar1}).
\end{lyxlist}

\medskip{}

Facts (i) through (v) were proved without using any results from the
theory of analytic capacity.

\end{fact*}
Let $\mathcal{K}^{*}$ be the set of $K\in\mathcal{K}$ such that:
$\mathrm{int}(K)\neq\emptyset$ and there is a $z\in(\partial K)_{\text{I}}$
that is not an ep (equivalently, Je) point. For $K\in\mathcal{K}^{*}$,
the proof that each non-ep (or non-Je) point of $(\partial K)_{\text{I}}$
is a peak point uses the Curtis Peak Point Criterion \cite{Bachar1,Curtis}
whose proof requires results from the theory of analytic capacity.

Thus, \cite{Bachar1} contains the complete proof that: $\mathcal{P}(P(K))=\partial K$
for all $K\in\mathcal{K}.$

We now proceed with the new capacity-free proofs of the results stated
in the Abstract.

\section{Notations and properties of the $\alpha$-th root Function}

The $\alpha$-th root function is an essential tool for the results
proved in what follows. The properties listed below are standard results
from complex variable theory.

For $r>0$ and $z_{0}\in\mathbb{C}$, the closed disk, $\left\{ z\in\mathbb{C}:|z-z_{0}|\leq r\right\} $,
is denoted by $\overline{D}(z_{0};r)$; the open disk, $\left\{ z\in\mathbb{C}:|z-z_{0}|<r\right\} $,
is denoted by $D(z_{0};r)$; the circle, $\left\{ z\in\mathbb{C}:|z-z_{0}|=r\right\} $,
is denoted $C(z_{0};r)$; the topological boundary of a set $K$ is
denoted by $\partial K$; clearly, $\partial\overline{D}(z_{0};r)=\partial D(z_{0};r)=C(z_{0};r)$.

\medskip{}

\begin{defn}
We define several basic concepts and their properties in the following
list:
\begin{itemize}
\item The \emph{exponential function} on $\mathbb{C}$, $\exp(z)=\sum_{n=0}^{\infty}\frac{z^{n}}{n!}$,
is an entire function; its restriction to the infinite open strip,
\[
S_{\left(\text{-}\pi\text{,}\pi\right)}=\left\{ z=x+iy\in\mathbb{C}:(-\infty<x<\infty),(-\pi<y<\pi)\right\} 
\]
is a conformal map from $S_{\left(\text{-}\pi\text{,}\pi\right)}$
onto $\mathbb{C}\sm(-\infty,0]$ with conformal inverse, $\log:=\exp^{\text{-1}}$,
that maps $\mathbb{C}(-\infty,0]$ onto $S_{\left(\text{-}\pi\text{,}\pi\right)}$.
For any $z_{1}$ and $z_{2}$ in $S_{\left(\text{-}\pi\text{,}\pi\right)}$,
$\exp(z_{1}+z_{2})=\exp(z_{1})\exp(z_{2})$, and for any $z_{1}$and
$z_{2}$ in $\mathbb{C}\sm(-\infty,0]$, $\log(z_{1}z_{2})=\log(z_{1})+\log(z_{2})$.
\item The \emph{argument function}, $\arg$, is defined for $z\in\bbC\sm(-\infty,0]$\ by
$\arg(z)=\mathrm{Im}(\log(z))$, and is continuous. 
\item The \emph{$\alpha$-th root function}, $\mathbb{Z}^{\alpha}$, for
$\alpha\in(0,1]$, is defined by $\mathbb{Z}^{\alpha}(z)\equiv(\mathrm{exp}\circ m_{\alpha}\circ\mathrm{log})(z)$,
for $z\in\mathbb{C}\sm(-\infty,0]$, and $\mathbb{Z}^{\alpha}(0)=0$,
where $m_{\alpha}(z)=\alpha z$, for $z\in\mathbb{C}$. In other words,
$\mathbb{Z}^{\alpha}=\exp\circ m_{\alpha}\circ\log$. Now $\mathbb{Z}^{\alpha}$
is conformal on $\mathbb{C}\sm(-\infty,0]$, and is a homeomorphism
on $\mathbb{C}\ensuremath{(-\infty,0)}$. 
\item If we define the \emph{open cone}, 
\[
C_{\alpha}\boldsymbol{\equiv}\left\{ w\in\mathbb{C}\sm(-\infty,0]:|\arg(w)|<\alpha\pi\right\} ,
\]
then $\mathbb{Z}^{\alpha}$ maps $\mathbb{C}\sm(-\infty,0]$ conformally
onto $C_{\alpha}$, $(\mathbb{Z}^{\alpha})^{\text{-1}}$ maps $C_{\alpha}$
conformally onto $\mathbb{C}\sm(-\infty,0]$, $\mathbb{Z}^{\alpha}$maps
$\mathbb{C}\sm(-\infty,0)$ homeomorhically onto $\left\{ 0\right\} \cup C_{\alpha}$,
and ($\mathbb{Z}^{\alpha}$)$^{\text{-1}}$ maps $\left\{ 0\right\} \cup C_{\alpha}$
homeomorphically onto $\mathbb{C}\sm(-\infty,0)$. 
\item Note that 
\begin{eqnarray*}
(\mathbb{Z}^{\alpha})^{\text{-1}} & = & (\exp\circ m_{\alpha}\circ\log)^{\text{-1}}\\
 & = & (\log)^{\text{-1}}\circ(m_{\alpha})^{\text{-1}}\circ(\exp)^{\text{-1}}\\
 & = & \exp\circ m_{\frac{1}{\alpha}}\circ\log
\end{eqnarray*}
is defined on $C_{\alpha}$ and $(\mathbb{Z}^{\alpha})^{\text{-1}}(0)=0$. 
\item Note that $C_{1}=\mathbb{C}\sm(-\infty,0]$.
\item For each $z\in\mathbb{C}\sm(-\infty,0]$, we have the unique representation,
$z=|z|\exp(i(\arg(z)))$, where $-\pi<\arg(z)<\pi$. Furthermore,
$\log(z)=\log(|z|)+i(\arg(z))$. 
\item For $\alpha\in(0,1]$ and $z\in\mathbb{C}\sm(-\infty,0)$, $\mathbb{Z}^{\alpha}(z)=0$
iff $z=0$, and for $z\neq0$, $\mathbb{Z}^{\alpha}(z)=|z|^{\alpha}\exp(i\alpha(\arg(z))$. 
\item For $\alpha\in(0,1]$ and $z\in C_{\alpha}$, $(\mathbb{Z}^{\alpha})^{\text{-1}}(z)=0$
iff $z=0$, and for $z\neq0$, $(\mathbb{Z}^{\alpha})^{\text{-1}}(z)=|z|^{\frac{1}{\alpha}}\exp(i\frac{1}{\alpha}\arg(z))$.
\item For $0\leq\rho$ and $\rho<r$, define the \emph{open deleted annulus},
\begin{eqnarray*}
\boldsymbol{} &  & \mathrm{And}(0;\rho,r)\\
 &  & \equiv\left\{ z\in\bbC\sm\left\{ 0\right\} :\rho<|z|<r,-\pi<\arg(z)<\pi\right\} \\
 &  & \subset\mathbb{C}\sm(-\infty,0]
\end{eqnarray*}

\item For each $0<\delta<1$ and $0<\theta_{0}<\pi$, define \emph{the arc
along the circle}, $C(0;\delta)$, from $\delta\exp(-i\theta_{0})$
to $\delta\exp(i\theta_{0})$ by 
\begin{eqnarray*}
 &  & \mathrm{arc}[\delta\exp(-i\theta_{0}),\delta\exp(i\theta_{0})]\\
 &  & =\left\{ \delta\exp(i\theta):-\theta_{0}\leq\theta\leq\theta_{0}\right\} .
\end{eqnarray*}
For each $\alpha\in(0,1]$,
\begin{eqnarray*}
 &  & \mathbb{Z}^{\alpha}(\mathrm{arc}[\delta\exp(-i\theta_{0}),\delta\exp(i\theta_{0})])\\
 &  & =\mathrm{arc}[\delta^{\alpha}\exp(-i\alpha\theta_{0}),\delta\exp(i\alpha\theta_{0})]
\end{eqnarray*}
Furthermore, as $\alpha\searrow0$, $\delta^{\alpha}\nearrow1$.
\end{itemize}
\end{defn}
\medskip{}

\begin{rem*}
For later use in the construction of infinite products of functions
in $P(K)$, we need to examine the image of ($\overline{D}$($\frac{1}{2}$;
$\frac{1}{2}$) under the map $\mathbb{Z}^{\alpha}$.
\end{rem*}
\medskip{}

\begin{lem}
\emph{(``Teardrop'' Lemma)} \\
For $\alpha\in(0,1)$:
\begin{lyxlist}{00.00.0000}
\item [{\emph{(2.2.1)}}] $\mathbb{Z}^{\alpha}((\overline{D}(\frac{1}{2};\frac{1}{2}))$
is compact.
\item [{\emph{(2.2.2)}}] $\mathbb{Z}^{\alpha}(\partial(\overline{D}(\frac{1}{2};\frac{1}{2})))$
a Jordan curve.
\item [{\emph{(2.2.3)}}] $\mathbb{Z}^{\alpha}(\partial(\overline{D}(\frac{1}{2};\frac{1}{2})))=\left\{ 0\right\} \cup\left\{ (\cos\theta)^{\alpha}\exp(i\alpha\theta):-\frac{\pi}{2}<\theta<\frac{\pi}{2}\right\} $.
We put $t_{\alpha}\boldsymbol{\equiv}\mathbb{Z}^{\alpha}(\partial(\overline{D}(\frac{1}{2};\frac{1}{2})))$.
The curve $t_{\alpha}$, looks like the boundary of a teardrop.
\item [{\emph{(2.2.4)}}] $t_{\alpha}=\mathbb{Z}^{\alpha}(\partial(\overline{D}(\frac{1}{2};\frac{1}{2}))\subset\left\{ 0,1\right\} \cup(C_{\alpha/2}\cap D(\frac{1}{2};\frac{1}{2}))$.
\item [{\emph{(2.2.5)}}] $\int t_{\alpha}=\int(\mathbb{Z}^{\alpha}((\overline{D}(\frac{1}{2};\frac{1}{2})))$
is open and simply-connected and $t_{\alpha}\cup\int t_{\alpha}\subset\left\{ 0,1\right\} \cup(C_{\alpha/2}\cap D(\frac{1}{2};\frac{1}{2}))$.
We put $T_{\alpha}\boldsymbol{\equiv}t_{\alpha}\cup\int t_{\alpha}$
and call it an \emph{$\alpha$-teardrop}. 
\item [{\emph{(2.2.6)}}] For all $0<\rho<\frac{1}{2}$ and all 0$<\delta<\frac{1}{2}$,
there exists 0$<\alpha_{\rho\text{,}\delta}<1$ such that for all
$0<\alpha\leq\alpha_{\rho\text{,}\delta}$, $\mathbb{Z}^{\alpha}$
maps $\overline{D}(\frac{1}{2};\frac{1}{2})\sm\overline{D}(0;\rho)$
into $D(1;\delta)$.
\item [{\emph{(2.2.7)}}] For all $0<\theta<\frac{\pi}{2}$, there exists
$0<\alpha_{\rho\text{,}\delta\text{,}\theta}\leq\alpha_{\rho\text{,}\delta}$
such that for all $0<\alpha\leq\alpha_{\rho\text{,}\delta\text{,}\theta}$,
(2.2.6) holds and $|\arg(z)|<\theta$ for every $z\in T_{\alpha}$.
\end{lyxlist}
\end{lem}
\medskip{}

\begin{proof}
(2.2.1) $\mathbb{Z}^{\alpha}$ is a homeomorphism on $\mathbb{C}\sm(-\infty,0)$,
so $\mathbb{Z}^{\alpha}((\overline{D}(\frac{1}{2};\frac{1}{2}))$
is compact.

(2.2.2) The function $h:\mathbb{T}\rightarrow\partial(\overline{D}(\frac{1}{2};\frac{1}{2})))$,
defined by $h(z)=\frac{1}{2}(1+z)$, $z\boldsymbol{\in}\mathbb{T}$,
is a homeomorphism, so $\partial(\overline{D}(\frac{1}{2};\frac{1}{2}))$
is a Jordan curve. Since $\mathbb{Z}^{\alpha}$ is a homeomorphism,
and since a homeomorphic images of a Jordan curve is also a Jordan
curve, the result follows.

(2.2.3) For $0\in\partial(\overline{D}(\frac{1}{2};\frac{1}{2}))$,
$\mathbb{Z}^{\alpha}(0)=0$. For $z\in\partial(\overline{D}(\frac{1}{2};\frac{1}{2}))\sm\left\{ 0\right\} $,
$z=(\cos\theta)\exp(i\theta)$, where $\theta=\arg(z)\in(-\frac{\pi}{2},\frac{\pi}{2})$.
This follows from the well-known fact that for any three distinct
points on a circle, two of which are at the opposite ends of a diameter
line, the triangle formed by using these points as vertices is a right
triangle. Thus, 
\begin{eqnarray*}
\mathbb{Z}^{\alpha}(z) & = & \mathbb{Z}^{\alpha}((\cos\theta)\exp(i\theta))\\
 & = & |\cos\theta|^{\alpha}\exp(i\alpha\theta)\\
 & = & (\cos\theta)^{\alpha}\exp(i\alpha\theta)
\end{eqnarray*}
Thus, (2.2.3) is proved.

(2.2.4) We know $\mathbb{Z}^{\alpha}(0)=0$. For $z\in\partial(\overline{D}(\frac{1}{2};\frac{1}{2}))\sm\left\{ 0\right\} $,
we proceed as follows. Define $f_{_{\alpha}\text{: }}(-\frac{\pi}{2},\frac{\pi}{2})\rightarrow\partial(\overline{D}(\frac{1}{2};\frac{1}{2}))\sm\left\{ 0\right\} $
by $f_{\alpha}(\theta)=(\cos\alpha\theta)\exp(i\alpha\theta)$. This
map is one-to-one, continuous and onto. To show that $\mathbb{Z}^{\alpha}((\partial(\overline{D}(\frac{1}{2};\frac{1}{2}))\sm\left\{ 0\right\} )\subset D(\frac{1}{2};\frac{1}{2})$,
it suffices to show that $(\cos\theta)^{\alpha}<\cos\alpha\theta$,
for $0\leq\theta<\frac{\pi}{2}$, or, equivalently, that $g(\theta):=\frac{\text{(cos }\theta\text{)}^{\alpha}}{\text{cos }\alpha\theta}<1$
for 0$<\theta<\frac{\pi}{2}$. The latter is true if $g^{\prime}(\theta)<0$
for 0$<\theta<\frac{\pi}{2}$. But a straightforward calculation shows
that $g^{\prime}(\theta)=\frac{\text{- }\alpha\text{(cos }\theta\text{)}^{\alpha}\text{sin((1 - }\alpha\text{)}\theta\text{)}}{\text{(cos }\theta\text{)(cos }\alpha\theta\text{)}^{2}}<0$
since every term (except $-\alpha$) in the expression is positive
for 0$<\theta<\frac{\pi}{2}$, and $-\alpha<0$. Since $g(\theta)=g(-\theta)$,
the result also holds for $-\frac{\pi}{2}<\theta<0$. Finally, we
show that $\mathbb{Z}^{\alpha}((\partial(\overline{D}(\frac{1}{2};\frac{1}{2}))\sm\left\{ 0\right\} )\subset C_{\alpha\text{/2}}$.
Since $\mathbb{Z}^{\alpha}(\partial(\overline{D}(\frac{1}{2};\frac{1}{2})))=\left\{ 0\right\} \cup f_{_{\alpha}}((-\frac{\pi}{2},\frac{\pi}{2}))$,
we immediately see that $\arg(f_{\alpha}(\theta))=\arg((\cos\alpha\theta)\exp(i\alpha\theta))=\alpha\theta$
for $\theta\in(-\frac{\pi}{2},\frac{\pi}{2})$, hence $\alpha\theta\in(-\frac{\alpha}{2}\pi,\frac{\alpha}{2}\pi)$,
i.e., $\mathbb{Z}^{\alpha}((\partial(\overline{D}(\frac{1}{2};\frac{1}{2}))\sm\left\{ 0\right\} )\subset C_{\alpha\text{/2}}$.

(2.2.5) This follows from the Jordan curve theorem.

(2.2.6) The two points of intersection, $C(0;\delta)\cap C(\frac{1}{2};\frac{1}{2})$,
are $\delta\exp(-i\theta_{\delta})$ and $\delta\exp(i\theta_{\delta})$,
where $\theta_{\delta}$ satisfies $\cos\theta_{\delta}=\delta$ (see
(2.2.3)). To finish the proof, it suffices to show that there exist
$\alpha\in$(0, 1) such that $\mathbb{Z}^{\alpha}$ maps the arc,
$\mathrm{arc}[\delta\exp(-i\theta_{\delta}),\delta\exp(i\theta_{\delta})]$,
into $\overline{D}(1;\rho)$. But 
\begin{eqnarray*}
 &  & \mathbb{Z}^{\alpha}(\mathrm{arc}[\delta\exp(-i\theta_{\delta}),\delta\exp(i\theta_{\delta})])\\
 &  & =\mathrm{arc}[\delta^{\alpha}\exp(-i\alpha\theta_{\delta}),\delta\exp(i\alpha\theta_{\delta})]
\end{eqnarray*}
Since $\delta^{\alpha}\nearrow1$ as $\alpha\searrow0$, it follows
that there exists $\alpha_{\rho\text{,}\delta}$ sufficiently close
to 0 such that the distance from every point on $\mathrm{arc}[\delta^{\alpha}\exp(-i\alpha\theta_{\delta}),\delta\exp(i\alpha\theta_{\delta})]$
to 1 is less than $\rho$.

(2.2.7) This is clear from (2.2.6) plus the fact that $|\arg(z)|\leq\alpha$
for every $z\in T_{\alpha}$. Hence, by choosing $\alpha\leq\min\left\{ \theta,\alpha_{\rho\text{,}\delta}\right\} $,
the proof is complete.
\end{proof}

\section{Results on linear fractional transformations (lft's)}

For later use, we need the following lemma.

\medskip{}

\begin{lem}
For any pair, $(v_{1},v_{2})$, of distinct points on the boundary,
$\mathbb{T}$, of the closed unit disk, $\overline{D}(0;1)$, there
exists a lft, $f$, of $\overline{D}(0;1)$ onto $\overline{D}(0;1)$
such that $f(v_{1})=1$ and $f(v_{2})=-1$.
\end{lem}
\medskip{}

\begin{proof}
If $v_{1}$ and $v_{2}$ are at opposite ends of a diameter line of
$\mathbb{T}$, then a simple rotation will yield the required lft.
If not, then a rotation will yield a lft taking $v_{1}$ to 1. Thus,
we may assume $v_{1}=1$ and that $v_{2}\in\mathbb{T}\sm\left\{ 1,-1\right\} $.
Henceforth, we label $v_{2}$ as $v$.

Any three distinct non-collinear points in $\mathbb{C}$ uniquely
determines the circle passing through them. Let $\left\{ z_{1},z_{2},z_{3}\right\} $
and $\left\{ w_{1},w_{2},w_{3}\right\} $ be two sets of distinct
non-collinear points on the unit circle, $\mathbb{T}$, such that
the motions along $\mathbb{T}$ from $z_{1}$ to $z_{2}$ to $z_{3}$
(resp., $w_{1}$ to $w_{2}$ to $w_{3}$) are counterclockwise. By
solving for w in terms of z, the cross ratio

\[
\frac{w-w_{1}}{w_{1}-w_{2}}\frac{w_{2}-w_{3}}{w_{3}-w}=\frac{z-z_{1}}{z_{1}-z_{2}}\frac{z_{2}-z_{3}}{z_{3}-z}
\]
uniquely determines the lft, $f$, that is a self-homeomorhism of
$\mathbb{T}$ and $\overline{D}(0;1)$, is biholomorphic on $D(0;1)$,
and maps $z_{i}$ to $w_{i}$, $i=1,2,3$.

We choose $\left\{ z_{1},z_{2},z_{3}\right\} =\left\{ 1,v,-1\right\} $
and $\left\{ w_{1},w_{2},w_{3}\right\} =\left\{ 1,-1,\overline{v}\right\} $.
The cross ratio yields 
\[
w=f_{\text{cc}}(z)=\frac{z(1-\overline{v}d)+(1+\overline{v}d)}{z(1-d)+(1+d)},\textrm{ where }d=\frac{2(1+v)}{v-\overline{v}}.
\]

Since it is known that the most general homeomorphism from $\overline{D}(0;1)$
onto $\overline{D}(0;1)$ and from $\mathbb{T}$ onto $\mathbb{T}$
that is conformal from $D(0;1)$ onto $D(0,1)$ and maps 1 to 1 has
the form 
\[
f_{\alpha}(z)=\frac{z-\alpha}{\overline{\alpha}z-1}\frac{\overline{\alpha}-1}{1-\alpha},\textrm{ for }|\alpha|<1,
\]
we must have $f=f_{\alpha}$ for some $\alpha$ with $|\alpha|<1$.
Hence, $f_{\alpha}(\alpha)=0=f(\alpha)$, from which it follows that
$\alpha=-\frac{1+\overline{v}d}{1-\overline{v}d}$. 
\end{proof}
\medskip{}

The following facts about linear fractional transformations are basic
and well known.

\medskip{}

\begin{lem}
For a lft $f$ the following conditions hold:
\begin{lyxlist}{00.00.0000}
\item [{\emph{(3.2.1)}}] The most general lft of the closed unit disk $\overline{D}(0;1)$
onto itself is $f(z)=c\frac{z-\alpha}{\overline{\alpha}z-1}$, where
$\alpha$ and $c$ are complex, $\big|c\big|=1$, $\big|\alpha\big|<1$,
and $z\in\overline{D}(0;1)$ .
\item [{\emph{(3.2.2)}}] $f(1)=1$ iff $c=-\frac{1\text{ - }\overline{\alpha}}{1\text{ - }\alpha}$
and $f(-1)=-1$ iff $c=-\frac{1\text{ + }\overline{\alpha}}{1\text{ + }\alpha}$.
\item [{\emph{(3.2.3)}}] $f(1)=1$ and $f(-1)=-1$ iff $c=-\frac{1\text{ - }\overline{\alpha}}{1\text{ - }\alpha}=-\frac{1\text{ + }\overline{\alpha}}{1\text{ + }\alpha}$
iff $\alpha$ is real.
\item [{\emph{(3.2.4)}}] $u\neq1$, $u\neq-1$, $\big|u\big|=1$, $f(u)=u$,
$f(1)=1$ and $f(-1)=-1$ iff $\alpha=0$ iff $f(z)=z$.
\item [{\emph{(3.2.5)}}] The most general lft of the disk $\bar{D}\left(\frac{1}{2};\frac{1}{2}\right)$
onto itself that maps $0$ to $0$ and $1$ to $1$ is $g\left(z\right)=\frac{\beta z}{\left(1-\beta\right)-\left(1-2\beta\right)z}$,
where $0<\beta<1$ (so $0<1-\beta<1$ also).
\item [{\emph{(3.2.6)}}] For every $w\in D\left(\frac{1}{2};\frac{1}{2}\right)$,
$g\left(w\right)\rightarrow0$ as $\beta\searrow0$ and $g\left(w\right)\rightarrow1$
as $\beta\nearrow1$. These two limits are easy to verify.
\end{lyxlist}
\end{lem}

\section{Established results}

The following established results are central to the main theorems
in this paper.

Before proceeding, recall that a \emph{Jordan curve} is the image,
$\Gamma$($\mathbb{T}$), of a homeomorphism, $\Gamma$: $\mathbb{T}\mapsto\mathbb{C}$.
where $\mathbb{T}$ is the unit circle. The inverse, $\Gamma^{\text{-1}}$:
$\Gamma$($\mathbb{T}$)$\mapsto\mathbb{T}$, is also a homeomorphism.
We use the notation $\int(\Gamma(\mathbb{T}))$ to denote the bounded
component of $\mathbb{C}\sm(\Gamma(\mathbb{T}))$ and $\mathrm{ext}(\Gamma(\mathbb{T}))$
to denote the unbounded component of $\mathbb{C}\sm(\Gamma(\mathbb{T}))$.

We now state Caratheodory's Extension theorem.

\medskip{}

\begin{thm}
\emph{(\cite{Caratheodory}, Caratheodory's Extension Theorem for
Jordan domains)}

If $U\neq\bbC$ is a simply-connected, open, and connected subset
of the complex plane and if $\partial U$ is a Jordan curve, $\Gamma$($\mathbb{T}$),
then the Riemann map (i.e., bijective, holomorphic, and with holomorphic
inverse), $f:U\mapsto D(0;1)$, from $U$ onto $D(0;1)$ has an extension,
f$_{\text{ext}}$, satisfying the following:
\begin{lyxlist}{00.00.0000}
\item [{\emph{(1)}}] $f_{\text{ext}}:\overline{U}(=U\cup\Gamma(\mathbb{T}))\mapsto\overline{D}(0;1)$
( = $D(0;1)\cup\mathbb{T}$) is a homeomorphism onto;
\item [{\emph{(2)}}] $f_{\text{ext}}:\Gamma(\mathbb{T})\mapsto\mathbb{T}$
is a homeomorphism onto;
\item [{\emph{(3)}}] $f_{\text{ext}}=f$ on $U$ and $f_{\text{ext}}$
is conformal from $U$ onto $D(0;1)$.
\end{lyxlist}
\end{thm}
\medskip{}

\begin{rem*}
Notice that $(f_{\text{ext}})^{\text{-1}}:\mathbb{T}\mapsto\Gamma(\mathbb{T})$
is a homeomorphism onto, but, in general, $(f_{\text{ext}})^{\text{-1}}$
and $\Gamma$ are not the same function. 
\end{rem*}
\medskip{}

\begin{cor}
Let $\Gamma(\mathbb{T})$ be any Jordan curve, and let $z_{0}\in\Gamma(\mathbb{T})$. 

Then there exists a homeomorphism $f:\Gamma(\mathbb{T})\cup\mathrm{int}(\Gamma(\mathbb{T}))\mapsto\overline{D}(0;1)$
such that: 
\begin{lyxlist}{00.00.0000}
\item [{\emph{(1)}}] $f$ restricted to $\Gamma(\mathbb{T})$ is a homeomorphism
onto $\mathbb{T}$
\item [{\emph{(2)}}] $f$ restricted to $\int(\Gamma(\mathbb{T}))$ is
a conformal map onto $D(0;1)$
\item [{\emph{(3)}}] $f(z_{0})=1$.
\end{lyxlist}
\end{cor}
\medskip{}

\begin{proof}
By Theorem 4.1, $f(z_{0})\in\mathbb{T}.$ There is a $u$ with $|u|=1$
such that $u(f(z_{0}))=1$. Thus, $F:=uf$ is the required function. 
\end{proof}
\medskip{}

\begin{thm}
\emph{(A more general version of Caratheodory's theorem)}

Let $g:D(0;1)\mapsto U$ be the inverse of the Riemann map and U is
open, connected and simply connected.

Then:

$g$ extends continuously to $G:\overline{D}(0;1))\mapsto\overline{U})$
if and only if the boundary of $U$ is locally connected.
\end{thm}
\medskip{}

Next, we state a result of Rudin.

\medskip{}

\begin{lem}
\emph{(\cite[Lemma 15.3, p.299]{Rudin})} For every set, $\left\{ u_{1},u_{2},...,u_{j}\right\} \subset\mathbb{C}$,

\begin{eqnarray*}
|(1+u_{1})(1+u_{2})...(1+u_{j})-1| & \leq & (1+|u_{1}|)(1+|u_{2}|)...(1+|u_{j}|)-1\\
 & \leq & \exp(|u_{1}|+|u_{2}|+...+|u_{j}|)-1.
\end{eqnarray*}

\end{lem}
\medskip{}

Next, we state and prove a property of the Euler function that plays
an indispensable role in the proof of the main theorem.

\medskip{}

\begin{thm}
\emph{(Properties of the Euler function)}

For each $0<r<1$, $\prod_{j=1}^{\infty}(1-z^{j})$ converges uniformly
on $\overline{D}(0;r)$ to a function, $f$, continuous on $\overline{D}(0;r)$
and holomorphic on $D(0;r)$, i.e., $f\in A(\overline{D}(0;r))$.\end{thm}
\begin{proof}
Since each $\prod_{j=1}^{n}(1-z^{j})\in A(\overline{D}(0;r))$ and
the latter is complete under the sup norm on

$\overline{D}(0;r)$, it suffices to show that $(\prod_{j=1}^{n}(1-z^{j}))_{n}$
is a Cauchy sequence in $A(\overline{D}(0;r))$.

(1) For all $k\geq1$, $n\geq1$, and $z\in\overline{D}(0;r)$, and
using Rudin's Lemma 4.4, 
\begin{eqnarray*}
\prod_{j=1}^{n+k}(1-z^{j})-\prod_{j=1}^{n}(1-z^{j}) & = & \prod_{j=1}^{n}(1-z^{j})\left[\prod_{j=n+1}^{n+k}(1-z^{j})-1\right]
\end{eqnarray*}
and so
\begin{eqnarray*}
\abs{\prod_{j=1}^{n}(1-z^{j})\left[\prod_{j=n+1}^{n+k}(1-z^{j})-1\right]} & = & \prod_{j=1}^{n}|1-z^{j}|\abs{\prod_{j=n+1}^{n+k}(1-z^{j})-1}\\
 & \leq & \prod_{j=1}^{n}|1-z^{j}|\left[\prod_{j=n+1}^{n+k}(1+|z^{j}|)-1\right]\\
 & \leq & \prod_{j=1}^{n}|1-z^{j}|\left[\prod_{j=n+1}^{n+k}(1+r^{j})-1\right]\\
 & \leq & \prod_{j=1}^{n}|1-z^{j}|\left[\exp(\sum_{j=n+1}^{n+k}r^{j})-1\right]\\
 & \leq & \prod_{j=1}^{n}|1-z^{j}|\left[\exp(\sum_{j=n+1}^{\infty}r^{j})-1\right]\\
 & = & \prod_{j=1}^{n}|1-z^{j}|\left[\exp(\frac{r^{n+1}}{1-r})-1\right]
\end{eqnarray*}

(2) For all $n\geq1$ and $z\in\overline{D}(0;r)$, 
\begin{eqnarray*}
|\prod_{j=1}^{n}(1-z^{j})-1| & \leq & \prod_{j=1}^{n}(1+|z^{j}|)-1\\
 & \leq & \prod_{j=1}^{n}(1+r^{j})-1\\
 & \leq & \exp(\sum_{j=1}^{n}r^{j})-1\\
 & \leq & \exp(\sum_{j=1}^{\infty}r^{j})-1\\
 & = & \exp(\frac{r}{1-r})-1
\end{eqnarray*}
Thus, 
\begin{eqnarray*}
|\prod_{j=1}^{n}(1-z^{j})|-1 & \leq & |\prod_{j=1}^{n}(1-z^{j})-1|\\
 & \leq & \exp(\frac{r}{1-r})-1
\end{eqnarray*}
so 
\[
|\prod_{j=1}^{n}(1-z^{j})|\boldsymbol{\leq}\exp(\frac{r}{1-r}).
\]
Putting (1) and (2) together, for all $k\geq1$, $n\geq1$, and $z\in\overline{D}(0;r)$,
\[
|\prod_{j=1}^{n+k}(1-z^{j})-\prod_{j=1}^{n}(1-z^{j})|\boldsymbol{\leq}\exp(\frac{r}{1-r})\left[\exp(\frac{r^{n+1}}{1-r})-1\right]\rightarrow0
\]
as $n\rightarrow\infty$.
\end{proof}
\medskip{}

\begin{thm}
For $z\in D(0;1)$, $\prod_{j=1}^{\infty}(1-z^{j})=\exp(-\sum_{j=1}^{\infty}\frac{1}{j}\frac{z^{j}}{1-z^{j}})$,
and for $0<r<1$, the series $\sum_{j=1}^{\infty}\frac{1}{j}\frac{z^{j}}{1-z^{j}}$
converges uniformly and absolutely on $\overline{D}(0;r)$.
\end{thm}
\medskip{}

\begin{proof}
For the latter assertion, let $z\in\overline{D}(0;r)$, so that $|z|\boldsymbol{\leq}r$.
By the ratio test, 
\begin{eqnarray*}
\frac{|a_{j+1}|}{|a_{j}|} & = & \frac{|\frac{1}{j+1}\text{ }\frac{z^{j+1}}{1-z^{j+1}}|}{|\frac{1}{j}\text{ }\frac{z^{j}}{1-z^{j}}|}\\
 & = & \frac{j}{j+1}\frac{|1-z^{j}|}{|1-z^{j+1}|}\abs z\\
 & \leq & \frac{j}{j+1}\frac{|1-z^{j}|}{|1-z^{j+1}|}r\\
 & \rightarrow & r<1\textrm{ as }j\rightarrow\infty
\end{eqnarray*}

Next, $\log(1-z)=-\sum_{j=1}^{\infty}\frac{z^{j}}{j}$ for $z\in D(0;1)$,
and the series converges uniformly and absolutely on $\overline{D}(0;r)$.
For all $k\geq1$, $z\in D(0;1)$ implies $z^{k}\in D(0;1)$, and
$z\in\overline{D}(0;r)$ implies $z^{k}\in\overline{D}(0;r)$, we
have $\log(1-z^{k})=-\sum_{j=1}^{\infty}\frac{(z^{k}\text{)}^{j}}{j}$
for $z\in D(0;1)$, and the series converges uniformly and absolutely
on $\overline{D}(0;r)$.

Note that for all $z\in D(0;1)$ and $k\geq1$, $1-z^{k}$ lies in
the positive right half-plane.

Therefore, using the fact that $\sum_{k=1}^{\infty}z^{k}=\frac{z}{1-z}$
we obtain 
\begin{eqnarray*}
\log(\prod_{k=1}^{n}(1-z^{k})) & = & \sum_{k=1}^{n}\log(1-z^{k})\\
 & = & \sum_{k=1}^{n}\left(-\sum_{j=1}^{\infty}\frac{(z^{k}\text{)}^{j}}{j}\right)\\
 & = & -(z+\frac{1}{2}z^{2}+\cdots+\frac{1}{j}z^{j}+\cdots)\\
 &  & -(z^{2}+\frac{1}{2}(z^{2})^{2}+\cdots+\frac{1}{j}(z^{2})^{j}+\cdots)\\
 &  & \cdots\\
 &  & -(z^{n}+\frac{1}{2}(z^{n})^{2}+\cdots+\frac{1}{j}(z^{n})^{j}+\cdots)\\
 & = & -(z^{1}+z^{2}+\cdots+z^{n})\\
 &  & -\frac{1}{2}((z^{1\text{ }})^{2}+(z^{2})^{2}+\cdots+(z^{n})^{2})\\
 &  & \cdots\\
 &  & -\frac{1}{j}((z^{1\text{ }})^{j}+(z^{2})^{j}+\cdots+(z^{n})^{j})\\
 &  & \cdots
\end{eqnarray*}
Using the fact that for $\abs z<1$ and $j\geq1$, $\sum_{k=1}^{n}\left(z^{j}\right)^{k}=\frac{z^{j}\left(1-\left(z^{j}\right)^{n}\right)}{1-z^{j}}$
we obtain

\begin{eqnarray*}
\log(\prod_{k=1}^{n}(1-z^{k})) & = & -\left(\frac{z(1-z^{n}\text{)}}{1-z}+\frac{1}{2}\frac{z^{2}(1-(z^{2}\text{)}^{n}\text{)}}{1-z^{2}}+\frac{1}{3}\frac{z^{3}(1-(z^{3}\text{)}^{n}\text{)}}{1-z^{3}}+\cdots\right.\\
 &  & \left.+\frac{1}{j}\frac{z^{j}(1-(z^{j}\text{)}^{n}\text{)}}{1-z^{j}}+\cdots\right)\\
 & = & -\sum_{j=1}^{\infty}\frac{1}{j}\frac{z^{j}(1-(z^{j}\text{)}^{n}\text{)}}{1-z^{j}}
\end{eqnarray*}

By the ratio test, for each $n\geq1$, $-\sum_{j=1}^{\infty}\frac{1}{j}\frac{z^{j}(1-(z^{j}\text{)}^{n}\text{)}}{1-z^{j}}$
converges uniformly and absolutely on $\overline{D}(0;r)$.

Furthermore, we claim that 
\[
\left\Vert -\sum_{j=1}^{\infty}\frac{1}{j}\frac{z^{j}(1-(z^{j}\text{)}^{n}\text{)}}{1-z^{j}}-(-\sum_{j=1}^{\infty}\frac{1}{j}\frac{z^{j}}{1-z^{j}})\right\Vert _{\overline{D}(0;r)}\rightarrow0\textrm{ as }n\rightarrow\infty.
\]

For $z\in\overline{D}(0;r)$, 
\begin{eqnarray*}
\abs{-\sum_{j=1}^{\infty}\frac{1}{j}\frac{z^{j}(1-(z^{j}\text{)}^{n}\text{)}}{1-z^{j}}-(-\sum_{j=1}^{\infty}\frac{1}{j}\frac{z^{j}}{1-z^{j}})} & = & \abs{\sum_{j=1}^{\infty}\frac{1}{j}\frac{z^{j}}{1-z^{j}}z^{jn}}\\
 & \leq & \sum_{j=1}^{\infty}\frac{1}{j}\frac{|z^{j}|}{|1-z^{j}|}|z|^{jn}\\
 & \leq & r^{n}\sum_{j=1}^{\infty}\frac{1}{j}\frac{|z^{j}|}{|1-z^{j}|}\\
 & \rightarrow & 0\textrm{ as }n\rightarrow\infty
\end{eqnarray*}
as claimed.

Consequently,
\[
\|\exp(-\sum_{j=1}^{\infty}\frac{1}{j}\frac{z^{j}}{1-z^{j}})-(-\sum_{j=1}^{\infty}\frac{1}{j}\frac{z^{j}(1-(z^{j}\text{)}^{n}\text{)}}{1-z^{j}}\|_{\overline{D}(0;r)}\rightarrow0\textrm{ as }n\rightarrow\infty.
\]

Next, 
\[
\exp(\log(\prod_{k=1}^{n}(1-z^{k})))=\prod_{k=1}^{n}(1-z^{k})=\exp\left(-\sum_{j=1}^{\infty}\frac{1}{j}\frac{z^{j}(1-(z^{j}\text{)}^{n}\text{)}}{1-z^{j}}\right).
\]

Finally, we have

(1) $\|\prod_{j=1}^{\infty}(1-z^{j})-\prod_{j=1}^{n}(1-z^{j})\|_{\overline{D}(0;r)}\rightarrow0$
as $n\rightarrow\infty$ (Theorem 4.5),

(2) $\|\exp(-\sum_{j=1}^{\infty}\frac{1}{j}\frac{z^{j}}{1-z^{j}})-(-\sum_{j=1}^{\infty}\frac{1}{j}\frac{z^{j}(1-(z^{j}\text{)}^{n}\text{)}}{1-z^{j}}\|_{\overline{D}(0;r)}\rightarrow0$
as $n\rightarrow\infty$, and

(3) $\prod_{j=1}^{n}(1-z^{j})=\exp\left(-\sum_{j=1}^{\infty}\frac{1}{j}\frac{z^{j}(1-(z^{j}\text{)}^{n}\text{)}}{1-z^{j}}\right)$,

which, together, imply $\prod_{j=1}^{\infty}(1-z^{j})=\exp(-\sum_{j=1}^{\infty}\frac{1}{j}\frac{z^{j}}{1-z^{j}})$.\end{proof}
\begin{cor}
For any $0\boldsymbol{\leq}x<1$, $\prod_{j=1}^{\infty}(1-x^{j})>0$.\end{cor}
\begin{proof}
Since $\prod_{j=1}^{\infty}(1-x^{j})=\exp(-\sum_{j=1}^{\infty}\frac{1}{j}\frac{x^{j}}{1-x^{j}})$,
and since $-\sum_{j=1}^{\infty}\frac{1}{j}\frac{x^{j}}{1-x^{j}}<0$,
we have $1>\exp(-\sum_{j=1}^{\infty}\frac{1}{j}\frac{x^{j}}{1-x^{j}})>0$.
\end{proof}
\medskip{}

We list various facts about circularly accessible points, topological
boundary points, and kissing disks that are used throughout the paper.

\medskip{}

\begin{defn}
Let $K\subset\mathbb{C}$ be compact, $\mathbb{C}\sm K$ be connected,
and $z\in\partial K$. Then
\begin{lyxlist}{00.00.0000}
\item [{(1)}] $z$ is a \emph{type I boundary point} (resp., \emph{type
II}) iff each (resp., some) open nghd of $z$ intersects (resp., does
not intersect) $\int(K)$. We denote the set of type I (resp., type
II) boundary points by $\partial K_{\text{I}}$ (resp., $\partial K_{\text{II}}$).
Note that $\partial K=\partial K_{\text{I}}\cup\partial K_{\text{II}}$
is a disjoint union. 
\item [{(2)}] $z$ is \emph{circularly accessible (ca)} iff there exists
$u_{z}\in\mathbb{T}$ and $1>r_{z}>0$ such that the closed disk $\overline{D}(c_{z};r_{z})$,
with center $c_{z}=z+r_{z}u_{z}$, satisfies $\overline{D}(c_{z};r_{z})\cap K=\left\{ z\right\} $
and $\overline{D}(c_{z};r_{z})\sm\left\{ z\right\} \subset\mathbb{C}\sm K$.
We denote the set of ca points by $\mathcal{P}_{\text{ca}}(K)$.
\item [{(3)}] Such a disk is called a \emph{kissing disk at $z$}, and
we use the notation $(\mathrm{kd})_{z,c_{z},r_{z}}\equiv\overline{D}(c_{z};r_{z})$.
\end{lyxlist}
\end{defn}
\medskip{}

\begin{thm}
\emph{(see \cite{Bachar2})} $\mathcal{P}_{\text{ca}}(K)$ is dense
in $\partial K$.\end{thm}
\begin{proof}
Let $\delta/2>0$ and $z\in\partial K$. There exists $z_{1}\in D(z;\delta/2))\sm K$.
Since $K$ is compact and $z_{1}\notin K$, there exists $z_{2}\in K$
such that $\delta_{1\text{ }}\equiv|z_{1}-z_{2}|=\mathrm{dist}(z_{1},K)>0$.
Note that $D(z_{1};\delta_{1})\subset\mathbb{C}\sm K$ and the open
line segment $(z_{1},z_{2})\subset\mathbb{C}\sm K$. Clearly, since
$z\in K$, $\mathrm{dist}(z_{1},K)\leq|z_{1}-z|<\delta/2$. Also,
$|z-z_{2}|\leq|z_{1}-z|+|z_{1}-z_{2}|<\delta/2+\delta/2=\delta$.
We claim that $z_{2}$ is a ca point. In fact, for every $w\in(z_{1},z_{2})$,
$\overline{D}(w;|w-z_{2}|)\cap K=\left\{ z_{2}\right\} $ and $\overline{D}(w;|w-z_{2}|)\sm\left\{ z_{2}\right\} \subset D(z_{1};\delta_{1})\subset\mathbb{C}\sm K$
as is easy to see.
\end{proof}

\section{Facts about connected topological spaces}
\begin{defn}
Let $(X,\mathbb{X})$ be a connected topological space, and let $\mathcal{C}$
$\subset\mathbb{X}$ be any covering of $X$ ($\mathbb{X}$ = topology
on $X$).
\begin{lyxlist}{00.00.0000}
\item [{(1)}] Two points $x$ and $y$ in $X$ are \emph{$\mathcal{C}$-simply
$n$-chained} iff there exists a finite sequence, $(U_{1},U_{2},...,U_{n})$
, of sets in $\mathcal{C}$ such that $x\in U_{1},\;y\in U_{n}$,
and $U_{i}\cap U_{j}\neq\emptyset$ iff $\big|i-j\big|\leq1$. We
term $(U_{1},U_{2},...,U_{n})$ a \emph{$\mathcal{C}$-simple $n$-chain
($\mathcal{C}$ snc)} and we say that $x$ is linked to $y$ by a
$\mathcal{C}$ snc.
\item [{(2)}] Two points $x$ and $y$ in $X$ are \emph{$\mathcal{C}$-weakly
$n$-chained} iff there exists a finite sequence, $(U_{1},U_{2},...,U_{n})$
, of sets in $\mathcal{C}$ such that $x\in U_{1},\;y\in U_{n}$,
and $U_{i}\cap U_{i+1}\neq\emptyset$ for $i=1,2,...,n-1$. We term
$(U_{1},U_{2},...,U_{n})$ a \emph{$\mathcal{C}$-weak $n$-chain
($\mathcal{C}$ wnc)} and we say that $x$ is linked to $y$ by a
$\mathcal{C}$ wnc.
\end{lyxlist}
\end{defn}
\medskip{}

\begin{rem}
If $(U_{1},U_{2},...,U_{n})$ is a $\mathcal{C}$ snc (resp., $\mathcal{C}$
wnc) that links $x$ to $y$, then $(U_{n},U_{n-1},...,U_{1})$ is
a $\mathcal{C}$ snc (resp., $\mathcal{C}$ wnc) that links $y$ to
$x$.
\end{rem}
\medskip{}

\begin{thm}
$(X,\mathbb{X})$ is a connected topological space iff for every covering,
$\mathcal{C}\subset\mathbb{X}$, and every pair of distinct points
$x$ and $y$ in $X$, there is a $\mathcal{C}$ wnc that links x
to y.\end{thm}
\begin{proof}
$\Rightarrow$

We argue the contrapositive. Suppose there exists a covering, $\mathcal{C}\subset\mathbb{X}$,
and there exists a pair of points $x\neq y$ in $X$ such that no
$\mathcal{C}$ wnc links $x$ to $y$. By Remark 5.2, it follows that
no $\mathcal{C}$ wnc links $y$ to $x$. Put 
\[
X'=\left\{ x'\in X:\textrm{ there is a }\mathcal{C}\textrm{ wnc linking }x\textrm{ to }x'\right\} 
\]
and put 
\[
Y'=\left\{ y'\in X:\textrm{ there is a }\mathcal{C}\textrm{ wnc linking }y\textrm{ to }y'\right\} .
\]
Thus, $y$ can not lie in $X'$ and $x$ can not lie in $Y'$. We
now prove that both $X'$ and $Y'$ are open. It suffices to only
show $X'$ is open because, by symmetry, the same proof will apply
to $Y'$.

Let $x'\in X'$ be arbitrary. By definition of $X'$, there is a $\mathcal{C}$
wnc, $(U_{1},U_{2},...,U_{n})$, linking $x$ to $x'$, and so $x\in U_{1}$,
$x'\in U_{n}$, and $U_{i}\cap U_{i+1}\neq\emptyset$ for $i=1,2,...,n-1$.
If we show that every point $x\lyxmathsym{\textquotedblright}$ of
$U_{n}$ is in $X'$, then since $U_{n}$is open, this will prove
that $x'$ is an interior point of $X'$, and hence that $X'$ is
open. But clearly, $x$ is linked to $x\lyxmathsym{\textquotedblright}$
by the same $\mathcal{C}$ wnc, $(U_{1},U_{2},...,U_{n})$, and so
$x\lyxmathsym{\textquotedblright}$ is in $X'$, as was to be proved.

To complete the proof, we will prove that $X$ is the disjoint union
of two non-empty open sets, and hence that $X$ is not connected.
If $X\sm(X'\boldsymbol{\cup}Y')$ is empty, we are done, so suppose
$X\sm(X'\boldsymbol{\cup}Y')$ is non-empty. Let $z\in X\sm(X'\boldsymbol{\cup}Y')$.
Since $\mathcal{C}$ is an open covering of $X$, there is an open
set $U$ in $\mathcal{C}$ such that $z\in U$. We'll show $U\subset X\sm(X'\boldsymbol{\cup}Y')$.
Suppose not. Then $U\cap X'\neq\emptyset$ or $U\cap Y'\neq\emptyset$.
Say that $U\cap X'\neq\emptyset$, and so there is a point $x'\in U\cap X'$.
Thus, there is a $\mathcal{C}$ wnc, $(U_{1},U_{2},...,U_{\text{n}})$,
linking $x$ to $x'$, and thus $x$ is linked to $z$ by the $\mathcal{C}$
wnc, $(U_{1},U_{2},...,U_{n},U)$, and so $z$ is in $X'$, a contradiction.
Hence, $z$ is an interior point of $X\sm(X'\boldsymbol{\cup}Y')$,
and so the pair, $(X'\boldsymbol{\cup}X\sm(X'\boldsymbol{\cup}Y'))$
and $Y'$, provides a non-trivial disconnection of $X$, as was to
be proved.

$\Leftarrow$ 

We argue the contrapositive. If $X$ is disconnected, let $X=U\boldsymbol{\cup}V$,
where both $U$ and $V$ are non-empty open sets that are disjoint.
Clearly, $\mathcal{C}=\left\{ U,V\right\} $ is an open covering of
$X$. If $u\in U$ and $v\in V$, then it is clear that there is no
$\mathcal{C}_{\text{wnc}}$ that links $u$ to $v$.\end{proof}
\begin{thm}
Let $(X,\mathbb{X})$ be a connected topological space. For every
$\mathcal{C}$ wnc, $(U_{1},U_{2},...,U_{n})$, that links $x$ to
$y$, where $x,y$ is any pair of points in $X$, there is an ordered
subset, $(U_{1},U_{i_{2}},U_{i_{3}},...,U_{i_{k}},U_{n})$, $1<i_{1\text{ }}<i_{2}<...<i_{k}<n$,
of $(U_{1},U_{2},...,U_{n})$ that is a $\mathcal{C}$ snc that links
$x$ to $y$.\end{thm}
\begin{proof}
If $n=1$ or 2, the result is trivial, so let $n=3$. If a $\mathcal{C}$
w3c, $(U_{1},U_{2},U_{3})$, is not a $\mathcal{C}$ s3c, then, clearly,
we must have $U_{1}\cap U_{3}\neq\emptyset$. Thus, $(U_{1},U_{3})$
is clearly a $\mathcal{C}$ s2c that links $x$ to $y$.

We now proceed by induction. Suppose the result is true for $n=1,2,3,...,N$.
Let $(U_{1},U_{2},...,U_{N},U_{N+1})$ be a $\mathcal{C}$ wnc that
is not a $\mathcal{C}$ snc. By defintion of a $\mathcal{C}$ snc,
there exists $i$ and $j$ in $\left\{ 1,2,...,N+1\right\} $ such
that $U_{\text{i}}\cap U_{\text{j}}\neq\emptyset$ and $\big|i-j\big|>1$.
We may assume that $i<j$. It follows that $(U_{1},U_{2},...,U_{i},U_{j},...,U_{N+1})$
is a $\mathcal{C}$ wnc that links $x$ to $y$ and whose length is
$\leq N$, and so induction applies.
\end{proof}

\section{Fundamental theorem on the existence of homeomorphic, conformal,
norm one peaking functions for $A(K)$}
\begin{defn}
Let $K\in\mathcal{K}$ with $\int(K)\neq\emptyset$, and let $z_{1}$
and $z_{2}$ be two distinct points in $\partial K_{\text{I}}\cap\mathcal{P}_{\text{ca}}(K)$.

Let $\Gamma$ be a Jordan path such that
\begin{enumerate}
\item $z_{\text{i}}\in\Gamma$, $i=1,2$;
\item $\Gamma\backslash\left\{ z_{1},z_{2}\right\} \subset\mathbb{C}\backslash K$;
\item $\int(\Gamma)\subset\mathbb{C}\backslash K$.
\end{enumerate}
We call such a path a kissing path with respect to $z_{1}$ and $z_{2}$.\end{defn}
\begin{thm}
Consider $K\in\mathcal{K}$ with $\int(K)\neq\emptyset$, and a distinct
pair $z_{1}$ and $z_{2}$ in $\partial K_{\text{I}}\cap\mathcal{P}_{\text{ca}}(K)$.

(a) For every such pair there exists a kissing path with respect to
$z_{1}$ and $z_{2}$.

(b) There exist $f\in A(K)$ such that:
\begin{enumerate}
\item $\|f\|=1$;
\item f is a homeomorphism on K and conformal on int(K);
\item $f(z_{1})=1$and $f(z_{2})=-1$;
\item $f(K)\subset\left\{ 1,-1\right\} \cup D(0;1)$, f$(K\backslash\left\{ z_{1},z_{2}\right\} \subset D(0,1)$.
\end{enumerate}

(c) There exist $g\in A(K)$ such that:
\begin{enumerate}
\item $\|g\|=1$;
\item g is a homeomorphism on K and conformal on int(K);
\item $g(z_{1})=1$and $g(z_{2})=0$;
\item $g(K)\subset\left\{ 1,0\right\} \cup D(\frac{1}{2};\frac{1}{2})$,
g$(K\backslash\left\{ z_{1},z_{2}\right\} )\subset D(\frac{1}{2},\frac{1}{2})$.
\end{enumerate}
\end{thm}
\begin{proof}
Part (a)

For $i=1,2$, the radii, $r_{z_{1}}$ and $r_{z_{2}}$, of the kissing
disks, $\overline{D}(c_{z_{i}};r_{z_{i}})$, $i=1,2$, may be shrunk
sufficiently so that they are disjoint. Let $w_{i}=z_{i}+2(r_{z_{i}})u_{z_{i}}$,
where $u_{z_{i}}=\frac{c_{z_{i}}-z_{i}}{\abs{c_{z_{i}}-z_{i}}}$,
be the ends of the diameter line of the kissing disk starting at $z_{i}$
and passing through the center, $c_{z_{i}}$.

It is a well-known fact that an open set $A\subset\bbR^{n}$ is connected
iff it is pathwise-connected iff it is arcwise-connected.

Thus since $\mathbb{C}\sm K$ is open and connected, there is a $f:\left[0,1\right]\rightarrow\mathbb{C}\sm K$
that is a homeomorphism onto its image and such that $f\left(0\right)=w_{1}$
and $f\left(1\right)=w_{2}$.

For $i=1,2$, put $S_{i}=f^{-1}\left(\overline{D}(c_{z_{i}};r_{z_{i}})\right)$.
Clearly $0\in S_{1}$ and $1\in S_{2}$. Moreover, $S_{i}$ is a closed,
compact proper subset of $\left[0,1\right]$ and $z_{i}$ does not
lie in $f\left(\left[0,1\right]\right)$.

Define $l=\sup\left\{ t\in S_{1}\right\} $ and $u=\inf\left\{ t\in S_{2}\right\} $.
Using the definitions of supremum and infimum and the facts that $f$
and $f^{-1}$ are homeomorphisms, it is a straightforward consequence
that:

(1) $\zeta_{1}=f\left(l\right)\in\partial\overline{D}(c_{z_{1}};r_{z_{1}})-\left\{ z_{1}\right\} $, 

(2) $\zeta_{2}=f\left(u\right)\in\partial\overline{D}(c_{z_{2}};r_{z_{2}})-\left\{ z_{2}\right\} $. 

Thus, if $h=f|_{\left[l,u\right]}$, then $h$ is homeomorphic with
its image, and the image is in the complement of $K\cup\overline{D}(c_{z_{1}};r_{z_{1}})\cup\overline{D}(c_{z_{2}};r_{z_{2}})-\left\{ \zeta_{1},\zeta_{2}\right\} $
and $h\left(l\right)=\zeta_{1}$ and $h\left(u\right)=\zeta_{2}$.

If we define $g\left(t\right)=f\left(l+\left(u-l\right)t\right)$,
then $g$ is a reparametrization of $h$.

Put $G=g\left[0,1\right]$.

There is a $\delta>0$ sufficiently small such that the open cover,
$\mathcal{C}=\left\{ D(\zeta;\delta):\zeta\in G\right\} $, of $G$
satisfies

(1') the closure of every member of $\mathcal{C}$ is in $\mathbb{C}\sm K$;

(2') $\overline{D}(\zeta_{1};\delta)\cap\overline{D}(\zeta_{2};\delta)=\emptyset$; 

(3') for $i=1,2$, $\overline{D}(c_{z_{i}};r_{z_{i}})\cap\overline{D}(\zeta_{i};\delta)$
is a proper subset of $\overline{D}(c_{z_{i}};r_{z_{i}})$. 

By compactness of $G$, there is a finite subcover, $\mathcal{C}_{\text{fin}}$
of $G$ that includes $D(\zeta_{i};\delta)$, $i=1,...,N$, where
$\zeta_{i}\in G-\left\{ \zeta_{1},\zeta_{2}\right\} $ for $i=3,...,N$. 

By Theorem 5.3, there is a $\mathcal{D}_{\text{fin}}$-weak $n$-chain,
$\mathcal{C}\mathcal{H}_{1}$, linking $\zeta_{1}$ to $\zeta_{2}$,
and by Theorem 5.4, there is an ordered subset, $\mathcal{C}\mathcal{H}_{2}\subset\mathcal{C}\mathcal{H}_{1}$
that is a $\mathcal{D}_{\text{fin}}$-simple $n$-chain that links
$\zeta_{1}$ to $\zeta_{2}$. Thus $\mathcal{C}\mathcal{H}_{2}=\left(D\left(\sigma_{1};\delta\right),...,D\left(\sigma_{n};\delta\right)\right)$,
where $\sigma_{i}\in G$, $i=1,...,n$, $\zeta_{1}\in D\left(\sigma_{1};\delta\right)$
and $\zeta_{2}\in D\left(\sigma_{n};\delta\right)$. Define the circle
$C\left(\sigma;\delta=\partial\bar{D}\left(\sigma;\delta\right)\right)$.
Clearly the sequence of circles, $S_{\textrm{circ}}=\left(C\left(\sigma_{1};\delta\right),...,C\left(\sigma_{n};\delta\right)\right)$,
is an ordered simple $n$-chain.

Now $C\left(\sigma_{1};\delta\right)\cap C\left(c_{z_{1}};r_{z_{1}}\right)$
consists of two points: $c_{1}^{1}$ and $c_{2}^{1}$, where $c_{1}^{1}$
is the first intersection met when proceeding counterclockwise from
$z_{1}$ along $C\left(c_{z_{1}};r_{z_{1}}\right)$.

Similarly, $C\left(\sigma_{n};\delta\right)\cap C\left(c_{z_{2}};r_{z_{2}}\right)$
consists of two points: $c_{1}^{n}$ and $c_{2}^{n}$, where $c_{1}^{n}$
is the first intersection met when proceeding counterclockwise from
$z_{2}$ along $C\left(c_{z_{2}};r_{z_{2}}\right)$.

We next construct a Jordan path.

Firstly, we construct a sequence of subarcs of the circles in $S_{\textrm{circ}}$.
Starting at $c_{1}^{1}$, proceed counterclockwise along $C\left(\sigma_{1};\delta\right)$
until we reach the intersection $c_{2}\in C\left(\sigma_{2};\delta\right)\cap C\left(\sigma_{1};\delta\right)$.
Let $\textrm{arc}\left(1,2\right)$ denote the path from $c_{1}^{1}$
to $c_{2}$ along $C\left(\sigma_{1};\delta\right)$. Proceeding inductively,
we obtain a seqeunce of arc paths, joined end to end, $\textrm{arc}\left(1,2\right)$,
$\textrm{arc}\left(2,3\right)$, ... , $\textrm{arc}\left(n-1,n\right)$.
Note that the last arc ends at $c_{2}^{n}$.

Secondly, we repeat this construction by starting at $c_{1}^{n}$
and ending at $c_{2}^{1}$ in order to obtain the sequence of $n$
arcs $\textrm{Arc}\left(n,n-1\right)$, ..., $\textrm{Arc}\left(2,1\right)$.

We next construct the desired Jordan path $S$ as follows: start at
$z_{1}$, go counterclockwise along $C\left(c_{z_{1}};r_{z_{1}}\right)$
to $c_{1}^{1}$, traverse $\textrm{arc}\left(1,2\right)$, $\textrm{arc}\left(2,3\right)$,
... , $\textrm{arc}\left(n-1,n\right)$ to $c_{2}^{n}$, traverse
counterclockwise along $C\left(c_{z_{2}};r_{z_{2}}\right)$ to $c_{1}^{n}$,
traverse the arcs $\textrm{Arc}\left(n,n-1\right),...,\textrm{Arc}\left(2,1\right)$
to $c_{2}^{1}$, and the go counterclockwise along $C\left(c_{z_{1}};r_{z_{1}}\right)$
to get to $z_{1}$. It is straightforward to construct (detail omitted)
a homeomorphism, $\Gamma:\mathbb{T}\rightarrow S$. Thus, $S=\Gamma\left(\mathbb{T}\right)$
is a Jordan curve.

Note that the winding number of $z\in\int(S)$ w.r.t. $\Gamma$ is
$+1$. For $i=1,2$, the curve contains $z_{i}$, $\Gamma\left(\mathbb{T}\right)\backslash\left\{ z_{1},z_{2}\right\} \subset\mathbb{C}\backslash K$,
and $\int(S)\subset\mathbb{C}\backslash K$, i.e., $\Gamma(\mathbb{T})$
is a kissing path with respect to $z_{1}$ and $z_{2}$.

Part (b)

For $z_{0}\in\int(\Gamma(\mathbb{T}))$, define the lft, $\varphi(z)=\frac{1}{z-z_{0}\text{ }}$,
$z\in\mathbb{C}\cup\left\{ \infty\right\} $ (extended complex plane).
Clearly, $\varphi\circ\Gamma$ is a homeomorphism from $\mathbb{T}$
onto $(\varphi\circ\Gamma)(\mathbb{T})$ and $(\varphi\circ\Gamma)(\mathbb{T})$
Jordan curve. Furthermore,

\begin{tabular}{l}
$K\sm\left\{ z_{1},z_{2}\right\} \subset\mathrm{ext}((\Gamma(\mathbb{T})),$\tabularnewline
$z_{\text{i}}\in\Gamma(\mathbb{T}),\;i=1,2,$\tabularnewline
$\varphi(K\sm\left\{ z_{1},z_{2}\right\} )\subset\varphi(\mathrm{ext}((\Gamma(\mathbb{T})))=\int((\varphi\circ\Gamma)(\mathbb{T})),$\tabularnewline
$w_{i}\boldsymbol{:=}\varphi(z_{i})\in(\varphi\circ\Gamma)(\mathbb{T}),\;i=1,2,$\tabularnewline
$\varphi(\int(\Gamma(\mathbb{T}))=\mathrm{ext}((\varphi\circ\Gamma)(\mathbb{T})),$\tabularnewline
$\partial(\int((\varphi\circ\Gamma)(\mathbb{T})))=(\varphi\circ\Gamma)(\mathbb{T}).$\tabularnewline
\end{tabular}

By Theorem 4.1, there exists $F:(\varphi\circ\Gamma)(\mathbb{T})\cup\int((\varphi\circ\Gamma)(\mathbb{T}))\mapsto\overline{D}(0;1)$
such that:

(1) $F:(\varphi\circ\Gamma)(\mathbb{T})\mapsto\mathbb{T}$ is a homeomorphism
onto;

(2) $F_{\big|\text{int}((\varphi\circ\Gamma)(\mathbb{T}))}$ is conformal
from $\int((\varphi\circ\Gamma)(\mathbb{T}))$ onto $D(0;1)$;

(3) $v_{i}:=F(w_{i})\in\mathbb{T}$.

\medskip{}

By Lemma 3.1, there exists a lft, $f$, of $\overline{D}(0;1)$ onto
$\overline{D}(0;1)$ such that $f(v_{1})=1$ and $f(v_{2})=-1$.

Thus, the composition, $f\circ F$, satisfies the aforementioned properties
of $F$, except that $(f\circ F)(v_{1})=1$ and $(f\circ F)(v_{2})=-1$.

From the properties of all the mappings discussed, we conclude that
the restriction to $K$ of the composition map, $g:=f\circ F\circ\varphi$,
is an element of $P(K)$, is a homeomorphism of $K$ into $\left\{ 1,-1\right\} \cup D(0;1)$,
maps $z_{1}$ to 1, maps $z_{2}$ to $-1$, and maps $K\sm\left\{ z_{1},z_{2}\right\} $
into $D(0,1)$. 

Part (c)

By (b), select $f\in A(K)$ with all the properties listed in (b).
Define $h(z)=\frac{1}{2}(1+z)$, $z\in\mathbb{C}$. Now $h$ is a
homeomorphism from $\overline{D}(0;1)$ onto $\overline{D}(\frac{1}{2};\frac{1}{2})$
and from $\mathbb{T}$ onto $\partial D(\frac{1}{2};\frac{1}{2})$,
and is conformal from $D(0;1)$ onto $D(\frac{1}{2};\frac{1}{2})$.
Define $g=h\circ f$. Clearly, $g\in A(K)$, and the properties under
(c) clearly follow from the properties of $f$ and $h$.
\end{proof}

\section{Main results}
\begin{defn}
For every distinct pair of points $u$ and $v$ in $\partial K_{\text{I}}\cap\mathcal{P}_{\text{ca}}(K)$,
Theorem 6.1 (b) assures the existence of an $g\in A(K)$ such that
$g$ is a homeomorphism from $K$ onto $\overline{D}(\frac{1}{2};\frac{1}{2})$,
is conformal from $\int(K)$ onto $D(\frac{1}{2};\frac{1}{2})$, $g(z)=0$
iff $z=u$, $g(z)=1$ iff $z=v$, $\|g\|=1$, and $g(K)\subset\left\{ 0,1\right\} \cup D(\frac{1}{2};\frac{1}{2})$.
We use the term, term, ``$g$ is a conformal homeomorphism relative
to $u$ and $v$''.
\end{defn}

\begin{defn}
Let $K\in\mathcal{K}$ with $\int(K)\neq\emptyset$, let $z_{0}\boldsymbol{\in}\partial K_{\text{I}}$,
let $(z_{n})_{n\geq1}$ be a distinct sequence in $\partial K_{\text{I}}\cap\mathcal{P}_{\text{ca}}(K)$
such that $|z_{n}-z_{0}|\searrow0$, and let $(f_{i})_{i}\subset M_{z_{0}}\cap A(K)$
be a sequence such that
\begin{enumerate}
\item $f_{i}(z_{j})=\delta_{i,j}$(Kronecker delta),
\item $\|f_{i}\|=1$ for all $i\geq1$,
\item for all $z\in K\backslash\left\{ z_{0},z_{1},z_{2},...,z_{j},...\right\} $,
$f_{i}(z))\subset C_{\frac{1}{2^{3}}}\cap D(\frac{1}{2};\frac{1}{2})$.
\end{enumerate}
We call a sequence having these properties a $(0,1)$ sequence in
$A(K)$ with respect to $z_{0}$ and $(z_{n})_{n\geq1}$.\end{defn}
\begin{lem}
For every $0<\alpha<\frac{1}{2}$, for every $0<\zeta<\frac{1}{2}$,
and for every $z_{0}\in\partial K$ distinct from $u$ and $v$, there
exists $g_{\alpha\text{,}\zeta}\in A(K)$ such that $g_{\alpha\text{,}\zeta}$
is a conformal homeomorphism relative to $u$ and $v$ and g$_{\alpha\text{,}\zeta}(z_{0})\in D(0;\zeta^{\frac{1}{\alpha}})$.
Moreover, $f_{\alpha\text{,}\zeta}:=\mathbb{Z}^{\alpha}\circ g_{\alpha\text{,}\zeta}$
is a conformal homeomorphism relative to $u$ and $v$ and $f_{\alpha\text{,}\zeta}(z_{0})\in D(0;\zeta)\cap D(\frac{1}{2};\frac{1}{2})$.\end{lem}
\begin{proof}
By Theorem 6.2 (c), there exists $g\in A(K)$ such that g is a conformal
homeomorphism relative to $u$ and $v$. Note that $g(z_{0})\in D(\frac{1}{2};\frac{1}{2})$.

Note that $\zeta$ is real, positive and 0$<\zeta<\frac{1}{2}$ ,
and that $\mathbb{Z}^{\frac{1}{\alpha}}$ is defined for all $\alpha>0$,
so $\zeta^{\frac{1}{\alpha}}$ = $\mathbb{Z}^{\frac{1}{\alpha}}$($\zeta$)$\searrow0$
as $\alpha\searrow0$.

By Lemma 3.2, part (3.2.6), there exists $\beta_{\alpha\text{,}\varsigma}$
sufficiently close to 0 such that the lft,

$F_{\alpha\text{,}\zeta\text{ }}(z)=\frac{\beta_{\alpha\text{,}\varsigma}z}{\text{(1 -}\beta_{\alpha\text{,}\varsigma}\text{ ) - (1 - 2}\beta_{\alpha\text{,}\varsigma})z}$,
$z\in\overline{D}(\frac{1}{2};\frac{1}{2})$, satisfies $F_{\alpha\text{,}\zeta\text{ }}(g(z_{0}))\in D(0;\zeta^{\frac{1}{\alpha}})$.
Clearly, $g_{\alpha\text{,}\zeta}:=F_{\alpha\text{,}\zeta\diamond\text{ }}g$
is the sought-after function stated in the theorem. Finally, since
$g_{\alpha\text{,}\zeta}(z_{0})\in D(0;\zeta^{\frac{1}{\alpha}})\cap D(\frac{1}{2};\frac{1}{2})$,
it has the form $g_{\alpha\text{,}\zeta}(z_{0})=re^{\text{i}\theta}$,
where $0<r<\zeta^{\frac{1}{\alpha}}$ and $0\leq|\theta|<\frac{\pi}{2}$.
Thus, we have $f_{\alpha\text{,}\zeta}(z_{0})=\mathbb{Z}^{\alpha}\circ g_{\alpha\text{,}\zeta}(z_{0})=\mathbb{Z}^{\alpha}(re^{\text{i}\theta})=r^{\alpha}e^{\text{i}\alpha\theta},$
so $|f_{\alpha\text{,}\zeta}(z_{0})|=|\mathbb{Z}^{\alpha}(re^{\text{i}\theta})=|r^{\alpha}e^{\text{i}\alpha\theta}|=r^{\alpha}<\zeta$.
Hence, $f_{\alpha\text{,}\zeta}(z_{0})\in D(0;\zeta)\cap D(\frac{1}{2};\frac{1}{2})$
as claimed. \end{proof}
\begin{lem}
We claim that
\begin{lyxlist}{00.00.0000}
\item [{\emph{(a)}}] For every conformal homeomorphism, $f$, relative
to $u$ and $v$, for every $z_{0}\in\partial K$ distinct from $u$
and $v$, for every $0<\epsilon<\frac{1}{2}$, for every $0<\delta<\frac{1}{2}$,
there exists $0<\alpha_{0}<\frac{1}{2}$ such that for every $0<\alpha\leq\alpha_{0}$,
$f_{\epsilon\text{, }\delta\text{,}\alpha}:=\mathbb{Z}^{\alpha}\circ f$
is a conformal homeomorphism relative to $u$ and $v$ satisfying
\[
f_{\epsilon\text{, }\delta\text{,}\alpha}(K\backslash D(u;\epsilon))\subset D(1;\delta).
\]

\item [{\emph{(b)}}] In addition to the hypotheses in (a), let $\theta$
be any number such that 0$\leq|\theta|<\frac{\pi}{2}$. Then there
exists $\alpha_{\theta}$ such that, for all $\alpha$ satisfying
0$<\alpha\leq$min\{$\alpha_{0}$, $\alpha_{\theta}$\}, the function
$f_{\epsilon\text{, }\delta\text{,}\alpha}$ in (a) also satisfies
$|arg(f_{\epsilon\text{, }\delta\text{,}\alpha}(z)|<\theta$ for all
$z\in K\backslash\left\{ z_{0}\right\} $.
\end{lyxlist}
\end{lem}
\begin{proof}
(a)

Since $f$ is a homeomorphism on $K$, it follows that $f(D(u;\epsilon))$
is a relatively open nhd of 0 in $f(K)$. Hence, theres exists $\rho_{0}>0$
such that the relatively open disk, $D(0;\rho_{0})\cap f(K)$, in
$f(K)$ satisfies $D(0;\rho_{0})\cap f(K)\subset f(D(u;\epsilon)$.
Furthermore, $f(K\backslash D(u;\epsilon))=\overline{D}(\frac{1}{2};\frac{1}{2})\backslash f(D(u;\epsilon))$.

By Lemma 2.2 (``teardrop'' lemma), part (2.2.6), for all 0$<\rho<\frac{1}{2}$
and all 0$<\delta<\frac{1}{2}$, there exists 0$<\alpha_{\rho\text{,}\delta}<\frac{1}{2}$
such that for all 0$<\alpha\leq\alpha_{\rho\text{,}\delta}$, $\mathbb{Z}^{\alpha}(\overline{D}(\frac{1}{2};\frac{1}{2})\backslash D(0;\rho))\subset D(1;\delta)$.

Since $\overline{D}(\frac{1}{2};\frac{1}{2})\backslash D(0;\rho_{0})\supset\overline{D}(\frac{1}{2};\frac{1}{2})\backslash f(D(u;\epsilon)$,
we have, for all 0$<\alpha\leq\alpha_{\rho_{0}\text{,}\delta}$,

\[
\mathbb{Z}^{\alpha}(\overline{D}(\frac{1}{2};\frac{1}{2})\backslash f(D(u;\epsilon))\subset\mathbb{Z}^{\alpha}(\overline{D}(\frac{1}{2};\frac{1}{2})\backslash D(0;\rho_{0}))\subset D(1;\delta).
\]
Setting $f_{\epsilon\text{, }\delta_{0}\text{,}\alpha}:=\mathbb{Z}^{\alpha}\circ f$,
we have: $f_{\epsilon\text{, }\delta_{0}\text{,}\alpha}$ is a conformal
homeomorphism relative to $u$ and $v$, and
\begin{eqnarray*}
f_{\epsilon\text{, }\delta\text{,}\alpha}(K\backslash D(u;\epsilon)) & = & \mathbb{Z}^{\alpha}\circ f(K\backslash D(u;\epsilon))\\
 & = & \mathbb{Z}^{\alpha}(f(K\backslash D(u;\epsilon)))\\
 & = & \mathbb{Z}^{\alpha}(\overline{D}(\frac{1}{2};\frac{1}{2})\backslash f(D(u;\epsilon))\subset D(1;\delta).
\end{eqnarray*}

(b)

This follows by application of Lemma 2.2 (``teardrop''), part (2.2.7).\end{proof}
\begin{lem}
For every distinct pair of points $u$ and $v$ in $\partial K_{\text{I}}\cap\mathcal{P}_{\text{ca}}(K)$,
for every $z_{0}\in\partial K_{\text{I}}$ distinct from $u$ and
$v$, for every $0<\zeta<\frac{1}{2}$, for every $0<\epsilon<\frac{1}{2}$,
for every $0<\delta<\frac{1}{2}$, and for every $\theta$ such that
$0\leq|\theta|<\frac{\pi}{2}$, there exists a conformal homeomorphism,
$f{}_{u,v}$, relative to $u$ and $v$ such that:
\begin{enumerate}
\item $\|f{}_{u,v}\|=1$
\item $f_{u,v}(z)=0$ iff $z=u$
\item $f_{u,v}(z)=$1 iff $z=v$
\item $f_{u,v}(z_{0})\in D(0;\zeta)$
\item $f_{u,v}(K\backslash D(u;\epsilon))\subset D(1;\delta)$
\item $|\mathrm{arg}(f_{u,v}(z)|<|\theta|$ for all $z\in K\backslash\left\{ z_{0}\right\} $.
\end{enumerate}
\end{lem}
\begin{proof}
All of the assertions are straightforward consequences of the facts
proven in Lemmas 7.3 and 7.4.\end{proof}
\begin{lem}
Let $K\in\mathcal{K}$ and $z_{0}\in(\partial K)_{\text{ni}}$ (=
the set of non-isolated points of $K$). 

Then:

Every peaking function at $z_{0}$ is non-constant, and the following
are equivalent:
\begin{lyxlist}{00.00.0000}
\item [{\emph{(A)}}] $z_{0}$ is a peak point for $A(K)$. 
\item [{\emph{(B)}}] There exists $f\in A(K)$ such that $f(z)=0$ iff
$z=z_{0}$, $f(K)\subset\left\{ 0\right\} \cup D(\frac{1}{2};\frac{1}{2})$,
and there exists $z\in K\backslash\left\{ z_{0}\right\} $ such that
$f(z)\in D(\frac{1}{2};\frac{1}{2})$. 
\item [{\emph{(C)}}] There exists $f\in A(K)$ such that $f(z)=0$ iff
$z=z_{0}$ , $f(K)\subset\left\{ 0\right\} \cup C_{\frac{1}{2}}$,
and there exists $z\in K\backslash\left\{ z_{0}\right\} $ such that
$f(z)\in C_{\frac{1}{2}}$.
\end{lyxlist}
\end{lem}
\begin{proof}
Since $z_{0}$ is non-isolated, it is clear from the definition of
peak point that every peaking function at $z_{0}$ must be non-constant.

(A)$\Rightarrow$(B)

There exists $g\in A(K)$ such that $g(z_{0})=1$ and $|g(z)|<1$
for all $z\in K\backslash\left\{ z_{0}\right\} $. Put $f=\frac{1}{2}(1-g)$.
Then $f(z_{0})=\frac{1}{2}(1-g(z_{0}))=0$, and for $z\in K\backslash\left\{ z_{0}\right\} $,
$|f(z)-\frac{1}{2}|=|-\frac{1}{2}g(z)|<\frac{1}{2}$, so $f(z)\in D(\frac{1}{2};\frac{1}{2})$,
and thus $f(K)\subset\left\{ 0\right\} \cup D(\frac{1}{2};\frac{1}{2})$.
Further, since $g$, and hence, $f$, are non-constant, there exists
$z\in K\backslash\left\{ z_{0}\right\} $ such that $f(z)\in D(\frac{1}{2};\frac{1}{2})$.

(B)$\Rightarrow$(C)

This is immediate since $\left\{ 0\right\} \cup D(\frac{1}{2};\frac{1}{2})\subset\left\{ 0\right\} \cup C_{\frac{1}{2}}$.

(C)$\Rightarrow$(A)

We are given $f\in A(K)$ such that $f(z)=0$ iff $z=z_{0}$ and $f(K)\subset\left\{ 0\right\} \cup C_{\frac{1}{2}}$.
The function, $g=\frac{1}{9\|\text{f}\|}f\in A(K)$, satisfies $g(z)=0$
iff $z=z_{0}$, and $g(K)\subset\left\{ 0\right\} \cup(D(0;\frac{1}{9})\cap C_{\frac{1}{2}})$.
Next, $h=\mathbb{Z}^{\frac{1}{2}}\circ g\in A(K)$, satisfies $h(z)=0$
iff $z=z_{0}$, and $h(K)\subset\left\{ 0\right\} \cup(D(0;\frac{1}{3})\cap C_{\frac{1}{4}})$.
Finally, the function $F=1-h\in A(K)$, satisfies $\|F\|=1$, $F(z)=1$
iff $z=z_{0}$, and $|$$F(z)|<1$ for all $z\in K\backslash\left\{ z_{0}\right\} $,
and so $z_{0}$ is a peak point for $A(K)$.\end{proof}
\begin{thm}
Let $K\in\mathcal{K}$ with $\int(K)\neq\emptyset$. For every $z_{0}\in\partial K_{\text{I}}$,
for every distinct sequence, $(z_{n})_{n\geq1}$, in $\partial K_{\text{I}}\cap\mathcal{P}_{\text{ca}}(K)$
such that $|z_{n}-z_{0}|\searrow0$, \ there exists a $(0,1)$ sequence
in $A(K)$ with respect to $z_{0}$ and $(z_{n})_{n\boldsymbol{\geq1}}$.\end{thm}
\begin{proof}
In \cite{Bachar1}, Theorem 5.1.ii), it is proved that every type
II boundary point of $K$ is a peak point for $A(K)$. Therefore,
we need only consider type I boundary points of $K$. Let $z_{0}\in\partial K_{\text{I}}$.

Step 1: Selection of various sequences and disks used in the proof.

1.1 Since $\mathcal{P}_{ca}(K)$ is dense in $\partial K$ (proved
in \cite{Bachar1}), we may select any distinct sequence, $(z_{n})_{n\geq1},$
in$\partial K_{\text{I}}$ such that $|z_{n}-z_{0}|\searrow0$.

1.2 Select a distinct positive sequence, $(\epsilon_{n})_{n\geq1}$,
such that $\epsilon_{n}\searrow0$ and such that the sequence, $(\overline{D}(z_{n};\epsilon_{n}))_{n\geq1}$,
of closed disks is pairwise disjoint.

1.3 Select the sequence, $(D(1;\delta^{n}))_{n\geq1}$, where $0<\delta<\frac{1}{2}$
is fixed. Note that $\sum_{n=1,n\neq j}^{\infty}\delta^{n}<\sum_{n=1}^{\infty}\delta^{n}=\frac{\delta}{1\text{ - }\delta}<\frac{1}{2}$
for all $j\geq1$.

1.4 Select a positive sequence, $(\theta_{n})_{n\geq1}$, where $\theta_{n}\searrow$,
0$<\theta_{n}<\frac{\pi}{2}$ and $\sum_{n=1}^{\infty}\theta_{n}<\frac{\pi}{8}$.

1.5 Select a distinct positive sequence, $(\zeta_{n})_{n\geq1}$,
such that $\zeta_{n}\searrow0$.

Step 2: Selection of various sequences of functions in $A(K)$ used
in the proof.

Fix a $j\geq1$. We shall select functions, $f_{n,j}\in A(K)$, for
each $n\neq j$ so that f$_{n,j}(z_{j})=1$ and $f_{n,j}(z_{n})=0$.
Specifically, we use Lemma 7.5 to do this by identifying $z_{j}$
with $v$ and $z_{n}$ with $u$. The result is the sequence, $(f_{i,n})_{n\neq j}$,
in $A(K)$ such that

(1) $\|f_{n,j}\|=1$

(2) $f_{n,j}(z)=0$ iff $z=z_{n}$

(3) $f_{n,j}(z)=1$ iff $z=z_{j}$

(4) $f_{n,j}(z_{0})\in D(0;\zeta_{n})$

(5) $f_{n,j}(K\backslash D(z_{n};\epsilon_{n}))\subset D(1;\delta^{n})$

(6) $|\mathrm{arg}(f_{n,j}(z)|<|\theta_{n}|=\theta_{n}$ for all $z\in K\backslash\left\{ z_{0}\right\} $.

Step 3: For each $j\geq1$, construction of a sequence of finite products
of the $f_{n,j}$'s.

Define the finite products, for each $j\geq1$,

(3.1) $\Pi_{p}^{j}=\prod_{n=1,n\neq j}^{p}f_{n,j}\in A(K)$ for each
$p\geq1$.

Define the infinite sequence of finite products, for each $j\geq1$,

(3.2) $(\Pi_{p}^{j})_{p\geq1}$.

Fix $j\geq1$. For each $N\geq1$ and all $k\geq1$,

(3.3) $\prod_{n=1,n\neq j}^{N+k}-\prod_{n=1,n\neq j}^{N}=\prod_{n=1,n\neq j}^{N}\left[\prod_{n=N+1,n\neq j}^{N+k}-1\right]$

Step 4: Proof that for each $j\geq1$, $(\Pi_{p}^{j})_{p\geq1}$ is
a Cauchy sequence in $A(K)$.

Let $\epsilon>0$. By 1.5, there exists $N_{\epsilon}$ such that
for all $m\geq0$, $0<\zeta_{N_{\epsilon}+m}<\epsilon$. Since (4)
of step 2 shows that $f_{N_{\epsilon},j}(z_{0})\in D(0;\zeta_{N_{\epsilon}})$,
we have that $f_{N_{\epsilon},j}^{\text{-1}}(D(0;\zeta_{N_{\epsilon}})$
is an open nghd of $z_{0}$. Hence, for all

$z\in f_{N_{\epsilon},j}^{\text{-1}}(D(0;\zeta_{N_{\epsilon}})$,
$f_{N_{\epsilon},j}(z)\in D(0;\zeta_{N_{\epsilon}})\subset D(0;\epsilon)$.
Therefore, for all $z\in f_{N_{\epsilon},j}^{\text{-1}}(D(0;\zeta_{N_{\epsilon}})$,
$|f_{N_{\epsilon},j}(z)|<\epsilon$.

For all $k\geq1$ and for all $z\in f_{N_{\epsilon},j}^{\text{-1}}(D(0;\zeta_{N_{\epsilon}})$,
\begin{eqnarray*}
\abs{(\prod_{n=1,n\neq j}^{N_{\epsilon}+k}-\prod_{n=1,n\neq j}^{N_{\epsilon}})(z))} & = & \abs{(\prod_{n=1,n\neq j}^{N_{\epsilon}})(z)}\left[\abs{\prod_{n=N+1,n\neq j}^{N_{\epsilon}+k}(z)-1)}\right]\\
 & = & \prod_{n=1,n\neq j}^{N_{\epsilon}}\abs{f_{n,j}(z)}\abs{\left[\prod_{n=N+1,n\neq j}^{N_{\epsilon}+k}f_{n,j}(z)-1\right]}\\
 & = & \abs{f_{N_{\epsilon},j}(z)}\prod_{n=1,n\neq j}^{N_{\epsilon}\text{ - 1}}\abs{f_{n,j}(z)}\abs{\left[\prod_{n=N+1,n\neq j}^{N_{\epsilon}+k}f_{n,j}(z)-1\right]}\\
 & < & \epsilon\left\Vert \left[\prod_{n=1,n\neq j}^{N_{\epsilon}-1}f_{n,j}\right]\left[\prod_{n=N+1,n\neq j}^{N_{\epsilon}+k}f_{n,j}-1\right]\right\Vert \\
 & < & 2\epsilon.
\end{eqnarray*}

Next, we let $z$ be an arbitrary element of $K\backslash f_{N_{\epsilon},j}^{\text{-1}}(D(0;\zeta_{N_{\epsilon}})$.
Note that there exists

$D(z_{0};\rho)\subset f_{N_{\epsilon},j}^{-1}(D(0;\zeta_{N_{\epsilon}})$.
Hence, $K\backslash f_{N_{\epsilon},j}^{-1}(D(0;\zeta_{N_{\epsilon}})\subset K\backslash D(z_{0};\rho)$.
By (1.2) of step 1, there exist $N_{\rho}$ such that for all $p\geq N_{\rho}$,
and all $n\geq p$, $D(z_{n};\epsilon_{n})\subset D(z_{0};\rho)$.
Thus, $K\backslash f_{N_{\epsilon},j}^{-1}(D(0;\zeta_{N_{\epsilon}})\subset K\backslash D(z_{0};\rho)\subset K\backslash D(z_{n};\epsilon_{n})$
for all $n\geq p$.

Thus, for all $z\in K\backslash f_{N_{\epsilon},j}^{\text{-1}}(D(0;\zeta_{N_{\epsilon}})$
and all $n\geq p$, $z\in K\backslash D(z_{n};\epsilon_{n})$, so
by (5) of step 2,

$f_{n,j}(z)\in D(1;\delta^{n})$, and so {*} $|(f_{n,j}(z)-1|<\delta^{n}$
for all $n\geq p$ {*} .

Therefore, for $k\geq1$,

\begin{eqnarray*}
\abs{\left(\prod_{n=1,n\neq j}^{p+k}-\prod_{n=1,n\neq j}^{p}\right)(z))} & = & \abs{\left(\prod_{n=1,n\neq j}^{p}\right)(z))}\abs{\left[\prod_{n=p+1,n\neq j}^{p+k})(z)-1\right]}\\
 & = & \prod_{n=1,n\neq j}^{p}|f_{n,j}(z)|\abs{\left[\prod_{n=p+1,n\neq j}^{p+k}f_{n,j}(z)-1\right]}\\
 & \leq & \left[\prod_{n=1,n\neq j}^{p}\big\| f_{n,j}(z)\big\|\right]\abs{\left[\prod_{n=p+1,n\neq j}^{p+k}f_{n,j}(z)-1\right]}\\
 & \leq & \abs{\prod_{n=p+1,n\neq j}^{p+k}f_{n,j}(z)-1}\\
 & = & \abs{\prod_{n=p+1,n\neq j}^{p+k}\left[1+\left(f_{n,j}(z)-1\right)\right]-1}\\
 & \leq & \prod_{n=p+1,n\neq j}^{p+k}\left[1+\abs{f_{n,j}(z)-1}\right]-1\textrm{ (by Lemma 4.4)}\\
 & \leq & \exp\left(\sum_{n=p+1,n\neq j}^{p+k}\abs{f_{n,j}(z)-1}\right)-1\textrm{ (by Lemma 4.4)}\\
 & \leq & \exp\left(\sum_{n=p+1,n\neq j}^{p+k}\delta^{n}\right)-1\textrm{ (by * above)}\\
 & \leq & \exp\left(\sum_{n=p+1,n\neq j}^{\infty}\delta^{n}\right)-1
\end{eqnarray*}

Since $\sum_{n=p+1,n\neq j}^{\infty}\delta^{n}$ is the tail end of
a convergent series of positive numbers, $\sum_{n=p+1,n\neq j}^{\infty}\delta^{n}\searrow0$
as $p\rightarrow\infty$.

Furthermore, since $\exp(x)\searrow1$ as $x\searrow0$, it follows
that there exist $p_{\epsilon}$ such that $\exp(\sum_{n=p_{\epsilon}+1,n\neq j}^{\infty}\delta^{n})-1<(1+\epsilon)-1=\epsilon$.

Thus, for $p=p_{\epsilon}$, for all $z\in K\backslash f_{N_{\epsilon},j}^{\text{-1}}(D(0;\zeta_{N_{\epsilon}})$,
and for all $k\geq1$,

$|(\prod_{n=1,n\neq j}^{p_{\epsilon}+k}-\prod_{n=1,n\neq j}^{p_{\epsilon}})(z))|<\epsilon$.

By combining this with the first part of the proof, and letting $n_{\text{ }\epsilon}=\max\left\{ N_{\epsilon},p_{\epsilon}\right\} $,
we have for all $n\geq n_{\epsilon}$, for all $k\geq1$, and for
all $z\in K$, $\abs{(\prod_{n=1,n\neq j}^{n_{\epsilon}+k}-\prod_{n=1,n\neq j}^{n_{\epsilon}})(z))}<2\epsilon$,

and so for each $j\geq1$, $(\Pi_{p}^{j})_{p\geq1}$ is a Cauchy sequence
in $A(K)$ as was to be proved.

Hence, there exists $f_{j}\in A(K)$ such that $\|f_{j}-\prod_{n=1,n\neq j}^{N}f_{n,j}\|\rightarrow0$
as $N\rightarrow\infty$.

Step 5: Properties of the $f_{j}$'s.

Proof that $f_{j}(z_{0})=0$:

Since uniform convergence implies pointwise convergence, we have

$(\prod_{n=1,n\neq j}^{N}f_{n,j})(z)=\prod_{n=1,n\neq j}^{N}f_{n,j}(z)\rightarrow f_{j}(z)$
as $N\rightarrow\infty$ for all $z\in K$.

Since $z_{p}\rightarrow z_{0}$ as $p\rightarrow\infty$ and since
$f_{j}$ is continuous on $K$, $f_{j}(z_{p})\rightarrow f_{j}(z_{0})$
as $p\rightarrow\infty$ . But by (2) of step 2, for all $n\neq j$,
$f_{n,j}(z_{n})=0$. Thus, for $N_{p}>p$, and for all $p$, 
\[
(\prod_{n=1,n\neq j}^{N_{p}}f_{n,j})(z_{p})=\prod_{n=1,n\neq j}^{N_{p}}f_{n,j}(z_{p})=0,
\]
and so as $p\rightarrow\infty$, $N_{p}\rightarrow\infty$, and hence
\[
0=(\prod_{n=1,n\neq j}^{N_{p}}f_{n,j})(z_{p})\rightarrow f_{j}(z_{p})
\]
so $f_{j}(z_{p})=0$ for all $p$, so $0=f_{j}(z_{p})\rightarrow f_{j}(z_{0})$
as $p\rightarrow\infty$, so $f_{j}(z_{0})=0$.

Proof that $f_{j}(z_{j})=1$:

By (3) of step 2, $f_{n,j}(z_{j})=1$ for all $n\neq j$. Hence, $1=(\prod_{n=1,n\neq j}^{N}f_{n,j})(z_{j})\rightarrow f_{j}(z_{j})$
as $N\rightarrow\infty$, so $f_{j}(z_{j})=1$.

Proof that for $z\in(z_{n})_{n\geq0}\backslash\left\{ z_{\text{j}}\right\} $,
$f_{j}(z)=0$:

By (2) of step 2, $f_{n,j}(z_{n})=0$ for all $n\neq j$. We already
have shown that $f_{j}(z_{0})=0$. For $p\geq1$, $p\neq j$, and
$N>p$, $0=(\prod_{n=1,n\neq j}^{N}f_{n,j})(z_{p})\rightarrow f_{j}(z_{p})$
as $N\rightarrow\infty$. Hence, $f_{j}(z_{p})=0$, as was to be shown.

Proof that $\|f_{j}\|=1$:

\begin{eqnarray*}
\|f_{j}\| & = & \|f_{j}-\prod_{n=1,n\neq j}^{N}f_{n,j}+\prod_{n=1,n\neq j}^{N}f_{n,j}\|\\
 & \leq & \|f_{j}-\prod_{n=1,n\neq j}^{N}f_{n,j}\|+\|\prod_{n=1,n\neq j}^{N}f_{n,j}\|\\
 & \leq & \|f_{j}-\prod_{n=1,n\neq j}^{N}f_{n,j}\|+\prod_{n=1,n\neq j}^{N}\|f_{n,j}\|\\
 & \leq & \|f_{j}-\prod_{n=1,n\neq j}^{N}f_{n,j}\|+1\\
 & \rightarrow & 1\textrm{ as \ensuremath{N\rightarrow\infty}},
\end{eqnarray*}
so $\|f_{j}\|\leq1$. But, $f_{j}(z_{j})=1$, so $\|f_{j}\|=1$.

Proof that for all $z\in K\backslash\left\{ z_{0},z_{1},z_{2},...,z_{j},...\right\} $,
$f_{i}(z))\in C_{\frac{1}{8}}\cap D(\frac{1}{2};\frac{1}{2})$:

For every $n\neq j$ and for all $z\in K\backslash\left\{ z_{0},z_{1},z_{2},...,z_{j},...\right\} $,
$0<|f_{n,j}(z)|<1$.

Hence, $0<\prod_{n=1,n\neq j}^{N}|f_{n,j}(z)|<1$ for all $N$, and
$\prod_{n=1,n\neq j}^{N}|f_{n,j}(z)|$ is strictly decreasing as $N\rightarrow\infty$.

Furthermore, $0<\prod_{n=1}^{N}(1-\delta^{n})<1$ for all $N$, $\prod_{n=1}^{N}(1-\delta^{n})$
is strictly decreasing as $N\rightarrow\infty$, and $\prod_{n=1}^{N}(1-\delta^{n})\searrow\prod_{n=1}^{\infty}(1-\delta^{n})$
by Corollary 4.7.

Because $z\neq z_{0}$, there exists a disk, $D(z_{0};\rho)$ that
is properly contained in $K$ and $z\notin\overline{D}(z_{0};\rho)$.

There exists $N_{\rho}$ such that for all $n\geq N_{\rho}$, $(\overline{D}(z_{n};\epsilon_{n})\subset D(z_{0};\rho)$.
Thus $z\in(K\backslash D(z_{n};\epsilon_{n})$, so by (5) of step
2, $f_{n,j}(z)\in D(1;\delta^{n})$. Hence, $1-\delta^{n}<|f_{n,j}(z)|<1$.

Thus, for all $k\geq1$, $0<\prod_{n=N_{\rho}}^{N_{\rho}+k}(1-\delta^{n})<\prod_{n=N_{\rho},n\neq j}^{N_{\rho}+k}|f_{n,j}(z)|<1$.
Multiply this inequality by $\prod_{n=1}^{N_{\rho}-1}(1-\delta^{n})$
to get 
\[
0<\prod_{n=1}^{N_{\rho}+k}(1-\delta^{n})<\left[\prod_{n=1}^{N_{\rho}-1}(1-\delta^{n})\right]\left[\prod_{n=N_{\rho},n\neq j}^{N_{\rho}+k}|f_{n,j}(z)|\right]<1,
\]
and so 
\[
0<\prod_{n=1}^{\infty}(1-\delta^{n})<\left[\prod_{n=1}^{N_{\rho}-1}(1-\delta^{n})\right]\left[\prod_{n=N_{\rho},n\neq j}^{N_{\rho}+k}|f_{n,j}(z)|\right]<1,
\]
and so 
\[
0<\frac{\prod_{n=1}^{\infty}(1-\delta^{n})}{\left[\prod_{n=1}^{N_{\rho}-1}(1-\delta^{n})\right]}<\prod_{n=N_{\rho},n\neq j}^{N_{\rho}+k}|f_{n,j}(z)|<1.
\]
Let $k\rightarrow\infty$ to get 
\[
0<\frac{\prod_{n=1}^{\infty}(1-\delta^{n})}{\left[\prod_{n=1}^{N_{\rho}-1}(1-\delta^{n})\right]}\leq\prod_{n=N_{\rho},n\neq j}^{\infty}|f_{n,j}(z)|<1.
\]
Multiply this inequality by $\prod_{n=1,n\neq j}^{N_{\rho}-1}|f_{n,j}(z)|$
to get 
\[
0<\prod_{n=1,n\neq j}^{\infty}|f_{n,j}(z)|<\prod_{n=1,n\neq j}^{N_{\rho}-1}|f_{n,j}(z)|<1,
\]
so $0<\prod_{n=1,n\neq j}^{\infty}|f_{n,j}(z)|<1$.

Since $\prod_{n=1,n\neq j}^{N}f_{n,j}(z)\rightarrow f_{j}(z)$ as
$N\rightarrow\infty$, we have 
\[
1>\abs{\prod_{n=1,n\neq j}^{N}f_{n,j}(z)}=\prod_{n=1,n\neq j}^{N}\abs{f_{n,j}(z)}\rightarrow|f_{j}(z)|
\]
as $N\rightarrow\infty$. But 
\[
\prod_{n=1,n\neq j}^{N}\abs{f_{n,j}(z)}\rightarrow\prod_{n=1,n\neq j}^{\infty}\abs{f_{n,j}(z)}>0,
\]
so $0<|f_{j}(z)|<1$.

Now $\sum_{n=1,n\neq j}^{N}\arg(f_{n,j}(z))=\arg(\prod_{n=1,n\neq j}^{N}f_{n,j}(z))\rightarrow\arg(f_{j}(z))$
as $N\rightarrow\infty$.

By (6) of step 2, we have $-\theta_{n}<\arg(f_{n,j}(z))<\theta_{n}$,
so 
\[
-\sum_{n=1}^{N}\theta_{n}<\sum_{n=1,n\neq j}^{N}\arg(f_{n,j}(z))<\sum_{n=1}^{N}\theta_{n}.
\]

Since $\sum_{n=1}^{N}\theta_{n}<\frac{\tau}{8}$ , we have $-\frac{\tau}{8}<\sum_{n=1,n\neq j}^{N}\arg(f_{n,j}(z))<\frac{\tau}{8}$,
and since

\[
\arg(\prod_{n=1,n\neq j}^{N}f_{n,j}(z))\rightarrow\arg(f_{j}(z))
\]
as $N\rightarrow\infty$, we have $-\frac{\tau}{8}\leq\arg(f_{j}(z))\leq\frac{\tau}{8}$.

Hence, $f_{i}(z))\in C_{\frac{1}{8}}\cap D(\frac{1}{2};\frac{1}{2})$
as was to be proved.\end{proof}
\begin{thm}
Let $K\in\mathcal{K}$. Every $z_{0}\in\partial K$ is a peak point
for $A(K)$. Hence, $\mathcal{P}(A(K))=\partial K$.\end{thm}
\begin{proof}
In \cite{Bachar1}, Theorem 5.1.ii), it is proved that every type
II boundary point of $K$ is a peak point for $A(K)$.

Therefore, we need only consider type I boundary points of $K$.

Let $z_{0}\in\partial K_{\text{I}}$.

We use the sequence of functions $(f_{i})_{i}\subset M_{z_{0}}\cap A(K)$
from Theorem 7.7.

Let $(a_{n})_{n}$ be any sequence of strictly decreasing positive
numbers such that $\sum_{i=1}^{\infty}a_{i}=1$.

(1) Consider the sequence, $(a_{i}f_{i})_{i}$, in $A(K)$, the associated
sequence, $(\sum_{i=1}^{N}a_{i}f_{i})_{N}$, of partial sums in $A(K)$,
and the sequence, $(\sum_{i=1}^{N}\|a_{i}f_{i}\|)_{N}$. With respect
to the latter, $\sum_{i=1}^{N}\|a_{i}f_{i}\|\leq\sum_{i=1}^{N}a_{i}\|f_{i}\|\leq\sum_{i=1}^{N}a_{i}\rightarrow1$
as $N\rightarrow\infty$.

In other words, $(\sum_{i=1}^{N}a_{i}f_{i})_{N}$ converges absolutely,
hence it converges by the well-known theorem that every absolutely
convergent series in a Banach space is convergent. Thus, there exists
$(\sum_{i=1}^{N}\|a_{i}f_{i}\|)_{N}$ such that $\|f-\sum_{i=1}^{N}a_{i}f_{i}\|\rightarrow0$
as $N\rightarrow\infty$.

Also, for $z\in K$, $(\sum_{i=1}^{N}a_{i}f_{i})(z)=\sum_{i=1}^{N}a_{i}f_{i}(z)\rightarrow f(z)$
as $N\rightarrow\infty$, and

$\mathrm{Re}((\sum_{i=1}^{N}a_{i}f_{i})(z))=\mathrm{Re}(\sum_{i=1}^{N}a_{i}f_{i}(z))=\sum_{i=1}^{N}a_{i}\mathrm{Re}f_{i}(z)\rightarrow\mathrm{Re}(f(z))$
as $N\rightarrow\infty$.

(2) Note that $f_{i}(z_{0})=0$ for all $i$.

(3) For all $k\geq1$, $\left(\sum_{i=1}^{N}a_{i}f_{i}\right)(z_{k})\rightarrow f(z_{k})$
as $N\rightarrow\infty$, But for $N>k$, $\left(\sum_{i=1}^{N}a_{i}f_{i}\right)(z_{k})=a_{k}>0$,
so $f(z_{k})>0$, and so $f(z_{k})\in C_{\frac{1}{8}}\cap D(\frac{1}{2};\frac{1}{2})$.

(4) 
\begin{eqnarray*}
\|f\| & = & \|(f-\sum_{i=1}^{N}a_{i}f_{i})+\sum_{i=1}^{N}a_{i}f_{i})\|\\
 & \leq & \|(f-\sum_{i=1}^{N}a_{i}f_{i})\|+\|\sum_{i=1}^{N}a_{i}f_{i})\|\\
 & \leq & \|(f-\sum_{i=1}^{N}a_{i}f_{i})\|+\sum_{i=1}^{N}a_{i}\\
 & < & \|(f-\sum_{i=1}^{N}a_{i}f_{i})\|+1\\
 & \rightarrow & 1\textrm{ as \ensuremath{N\rightarrow\infty},}
\end{eqnarray*}

so $\|f\|\leq1$.

(5) Let $z\in K\backslash\left\{ z_{0},z_{1},z_{2},...,z_{n},...\right\} $.
By Step 5 of Theorem 7.7, for all $i$, we have $f_{i}(z)\in C_{\frac{1}{2^{3}}}\cap D(\frac{1}{2};\frac{1}{2})$,
so $0<\mathrm{Re}(f_{\text{i}}(z))<1$. Also, $0<a_{i}\mathrm{Re}(f_{i}(z))<a_{i}<1$.

Hence, $0<\mathrm{Re}((\sum_{i=1}^{N}a_{i}f_{i})(z))=\mathrm{Re}(\sum_{i=1}^{N}a_{i}f_{i}(z))=\sum_{i=1}^{N}a_{i}\mathrm{Re}f_{i}(z)<\sum_{i=1}^{N}a_{i}<1$.
But $\sum_{i=1}^{N}a_{i}\mathrm{Re}f_{i}(z)$ is strictly increasing,
so $\sum_{i=1}^{N}a_{i}\mathrm{Re}f_{i}(z)\nearrow\sum_{i=1}^{\infty}a_{i}\mathrm{Re}f_{i}(z)\leq1$.
On the other hand, $\sum_{i=1}^{N}a_{i}\mathrm{Re}f_{i}(z)\rightarrow\mathrm{Re}(f(z))$
as $N\rightarrow\infty$. Thus, $0<\mathrm{Re}(f(z))=\sum_{i=1}^{\infty}a_{i}\mathrm{Re}f_{i}(z)\leq1$.

Therefore, by Lemma 7.6, part (C), $z_{0}$ is a peak point for $A(K)$.
\end{proof}

\section{The coincidence, with respect to $A(K)$, of the Bishop minimal boundary,
the set of peak points, the topological boundary of $K$, and the
Shilov boundary.}
\begin{defn}
$PF(A(K))$ denotes the set of peaking functions in $A(K)$, i.e.,
the set of all $f\in A(K)$ such that there exists $z\in K$ with
$f(z)=1$ and $|f(w)|<1$ for $w\in K\backslash\left\{ z\right\} $.

$NPF(A(K))=A(K)PF(A(K))$ denotes the set on non-peaking functions
in $A(K)$.

Clearly, $A(K)=PF(A(K))\cup NPF(A(K))$ and the union is disjoint.

For each $f\in A(K)$, $M_{f}:=\left\{ z\in K:|f(z)|=\|f\|\right\} $
is called the maximal set for $f$, i.e., the closed compact subset
of $K$ on which f attains its maximum modulus.

A closed subset, $S\subset K$, is a boundary for $A(K)$ in the Shilov
sense iff for each $f\in A(K)$, $M_{f}\cap S\neq\emptyset$.

A subset, $N\subset K$, is a boundary for $A(K)$ in the Bishop sense
iff for each $f\in A(K)$, $M_{f}\cap N\neq\emptyset$.\end{defn}
\begin{thm}
Let K$\in\mathcal{K}$.
\begin{lyxlist}{00.00.0000}
\item [{\emph{(a)}}] The intersection, $\Gamma(A(K))$, of all subsets
$S\subset K$ that are boundaries for $A(K)$ in the Shilov sense
is a boundary for $A(K)$ in the sense of Shilov, and is contained
in every boundary for $A(K)$ in the Shilov sense. $\Gamma(A(K))$
is called the Shilov boundary for $A(K)$.
\item [{\emph{(b)}}] The class of all $N\subset K$, that are boundaries
for $A(K)$ in the Bishop sense contains a smallest one, $M(A(K))$,
called the Bishop minimal boundary for $A(K)$, and $M(A(K))=\mathcal{P}(A(K))$,
the set of peak points for $A(K)$.
\item [{\emph{(c)}}] $\mathcal{P}(A(K))\subset\Gamma(A(K))$, i.e., every
peak point for $A(K)$ is contained in the Shilov boundary for $A(K)$.
\end{lyxlist}
\end{thm}
\begin{proof}
(a) Proofs can be found in \cite{Stout,Gamelin,Bonsall-Duncan,Garnett}.

(b) See \cite{Bishop}, Theorem 1.

(c) It suffices to prove that for every $S\subset K$ that is a boundary
for $A(K)$ in the Shilov sense, $\mathcal{P}(A(K))\subset S$. We
argue the contrapositive: Suppose there exists an $S\subset K$ that
is a boundary for $A(K)$ in the Shilov sense but that $\mathcal{P}(A(K))$
is not contained in $S$, i.e., there exists $z_{0}\in\mathcal{P}(A(K))$
such that $z_{0}\notin S$. Since $z_{0}$ is a peak point, there
exists $f\in A(K)$ such that $f(z_{0})=1$ and $|f(w)|<1$ for $w\in K\backslash\left\{ z_{0}\right\} $.
Clearly, $M_{f}$, the the maximal set for $f$, satisfies $M_{f}=\left\{ z_{0}\right\} $,
and so $M_{f}\cap S=\emptyset$, which contradicts the fact that $S$
is a boundary for $A(K)$ in the Shilov sense.\end{proof}
\begin{thm}
For all $K\in\mathcal{K}$, $\partial K=\mathcal{P}(A(K))=M(A(K))=\Gamma(A(K))$.\end{thm}
\begin{proof}
$\partial K=\mathcal{P}(A(K))$ by Theorem 7.8, and $M(A(K))=\mathcal{P}(A(K))$
by Theorem 8.2 (b), and so

$\partial K=\mathcal{P}(A(K))=M(A(K))$.

$\mathcal{P}(A(K))\subset\Gamma(A(K))$by Theorem 8.2 (c).

To finish the complete proof, we now show that $\Gamma(A(K))\subset\mathcal{P}(A(K))$($=\partial K$
as just shown).

By Theorem 8.2 (a), $\Gamma(A(K))$ is contained in every boundary
for $A(K)$ in the Shilov sense. Hence it suffices to show that $\mathcal{P}(A(K))$
is a boundary for $A(K)$ in the Shilov sense.

We already know that $\mathcal{P}(A(K))$ is a closed subset of $K$
since $\partial K=\mathcal{P}(A(K))$.

For every $f\in PF(A(K))$, there exists $z_{f}\in\mathcal{P}(A(K)=\partial K$,
and so $M_{f}=\left\{ z_{f}\right\} $ and $M_{f}\cap\mathcal{P}(A(K)\neq\emptyset$.

Next, let $f\in NPF(A(K))$ be arbitrary. Now $M_{f}$ is a non-empty
subset of $K$. Let $z\in M_{f}$. If $z\in\partial K$, then $z\in\mathcal{P}(A(K)$,
hence, $M_{f}\cap\mathcal{P}(A(K)\neq\emptyset$. On the other hand,
if $z\notin\partial K$, then $z\in\int(K)$, and hence there exist
an open, connected, simply-connected set $U$ contained in $\int(K)$
to which $z$ belongs. This means that $f$ attains its maximum modulus
(with respect to $\bar{U}$) at an interior point, $w\in U$, hence
by the classical maximal modulus theorem, $f$ has the constant value
$f(w)$ on $\bar{U}$). In particular, since $\partial U\subset\partial K$,
there is a point $\zeta\in\partial K=\mathcal{P}(A(K)$ such that
$\zeta\in M_{f}$. Thus, $M_{f}\cap\mathcal{P}(A(K)\neq\emptyset$.

Thus, for every $f\in A(K)$, $M_{f}\cap\mathcal{P}(A(K)\neq\emptyset$,
so $\mathcal{P}(A(K)$ is a boundary for $A(K)$ in the Shilov sense
as claimed.
\end{proof}

\section{The set of $z\in K$ such that there exists a bounded approximate
identity in $M_{z}$ coincides with $\partial K$.}

In what follows, it suffices to consider commutative normed algebras.
\begin{defn}
Let $A$ be a commutative normed algebra. A bounded approximate identity
(bai) for $A$ is a net $\left\{ e_{\alpha}\right\} _{\alpha\in I}\subset A$
satisfying, for some $k>0$ and for all $\alpha\in I$, $\left\Vert e_{\alpha}\right\Vert \leq k$,
and satisfying $\lim_{\alpha}xe_{\alpha}=x$ for any $x\in A$.\end{defn}
\begin{thm}[\cite{Cohen}: Cohen factorization theorem]
Let $A$ be any commutative Banach algebra. If $A$ has a bai, then
for any $f\in A$, there exists $g,h\in A$ such that $f=gh$ and
for any $\delta>0$, $h$ can be chosen to be in the closed ideal
generated by $f$ so that $\|h-f\|<\delta$.
\end{thm}

\begin{thm}
For all $K\in\mathcal{K}$, 
\[
BAI(A(K)):=\left\{ z\in K:\textrm{ there exists a bai for }M_{z}\right\} =\partial K.
\]
\end{thm}
\begin{proof}
We first show $BAI(A(K))\subset\partial K$:

Let $M_{z_{0}}$ have a bai. Suppose $z_{0}\in\int(K)$. The function,
$f(z)=z-z_{0}$, $z\in K$, is in $M_{z_{0}}$. By Theorem 9.2, there
exists $g,h\in M_{z_{0}}$ such that $f=gh$. Now $f,\,g$ and $h$
are holomorphic on $\int(K)$, and so we can take complex derivatives
at $z=z_{0}$ and obtain $1=f^{\prime}(z_{0})=g^{\prime}(z_{0})h(z_{0})+h^{\prime}(z_{0})g(z_{0})=0$,
a contradiction. Thus, $z_{0}\in\partial K$.

Finally, we show $\partial K\subset BAI(A(K))$:

Let $z_{0}\in\partial K$ be isolated. Define $e\in M_{z_{0}}$ by
$e(z_{0})=0$, and $e(z)=1$ for $z\in K\backslash{z_{0}}$. For $n=1,2,3,...$,
put $e_{n}=e$. It is clear that $(e_{n})_{n}$ is a bai for $M_{z_{0}}$.

Let $z_{0}\in\partial K$ be non-isolated. By Theorem 8.3, $z_{0}$
is a peak point for $A(K)$. By Lemma 7.6 (B), there exist $f\in A(K)$
such that $f(z)=0$ iff $z=z_{0}$, $f(K)\subset\left\{ 0\right\} \cup D(\frac{1}{2};\frac{1}{2})$,
and there exists $z_{1}\in K\backslash\left\{ z_{0}\right\} $ such
that $f(z_{1})\in D(\frac{1}{2};\frac{1}{2})$.

Select 0$<\epsilon<\frac{1}{2}$ such that for all $n\geq1$, $\frac{\epsilon}{2^{n}}<|f(z_{1})|$.
We next define the following two sequences of nbhds of 0 and 1, respectively:
$(D(0;\frac{\epsilon}{2^{n}})_{n}$ and $(D(1;\frac{\epsilon}{2^{n}})_{n}$.

By Lemma 2.2, part 2.2.6, for each $n\geq1$, there exists $0<\alpha_{n}<1$
such that for all 0$<\alpha\leq\alpha_{n},$ $\mathbb{Z}^{\alpha}$
maps $\overline{D}(\frac{1}{2};\frac{1}{2})\backslash D(0;\frac{\epsilon}{2^{n}})$
into $D(1;\frac{\epsilon}{2^{n}})$. If we define $f_{n}=\mathbb{Z}^{\alpha_{n}}\circ f$
for all $n\geq1$, then it is a straightforward verification that
$(f_{n})_{n}$ is a bai for $M_{z_{0}}$.
\end{proof}

\section{The set of $z\in K$ satisfying the Bishop $\frac{1}{4}-\frac{3}{4}$
property is equal to $\partial K$}
\begin{defn}
We say $z\in K$ satisfies the Bishop $\frac{1}{4}-\frac{3}{4}$ property
\cite{Bishop-deLeeuw} iff for every open nghd $U$ of $z$, there
exists $f\in A(K)$ such that $\left\Vert f\right\Vert _{K}\leq1$,
$f\left(z\right)$ is real and $f\left(z\right)>\frac{3}{4}$, and
$\abs{f\left(w\right)}<\frac{1}{4}$ for all $w\in K\backslash U$. \end{defn}
\begin{thm}
Let $K\in\mathcal{K}$ and $z\in K$. The following are equivalent; 
\begin{lyxlist}{00.00.0000}
\item [{\emph{10.2.1}}] $z\in\partial K$ 
\item [{\emph{10.2.2}}] $z$ is a peak point for $A(K)$ 
\item [{\emph{10.2.3}}] $z$ satisfies the Bishop $\frac{1}{4}-\frac{3}{4}$
property 
\end{lyxlist}
\end{thm}
\begin{proof}
10.2.1 iff 10.2.2: This is just Theorem 7.8. 

10.2.2 implies 10.2.3: Let $U$ be any open neighborhood of $z$.
and let $f$ be a peaking function at $z$. The function $g=\frac{1}{2}(1+f)$
is also a peaking function at $z$ with $\left\Vert g\right\Vert _{K}\leq1$,
$g\left(z\right)=1>\frac{3}{4}$. Furthermore, the image of $g$ is
contained in $\{1\}\bigcup D(\frac{1}{2};\frac{1}{2})$. The image,
$g(K\backslash U)$, is a compact subset of $\{1\}\bigcup D(\frac{1}{2};\frac{1}{2})$
disjoint from $\{1\}$. Thus, $\sup\left\{ \abs{w-0}:w\in g(K\backslash U)\coloneqq\rho<1\right\} $.
By taking a sufficiently high power, $g^{n}$, of $g$, it is clear
that $g^{n}$ satisfies $\left\Vert g^{n}\right\Vert _{K}\leq1$,
$g^{n}\left(z\right)=1>\frac{3}{4}$ and $\abs{g^{n}\left(w\right)}<\frac{1}{4}$
for all $w\in K\backslash U$ as was to be proved. 

10.2.3 implies 10.2.1: We prove the contrapositive., i.e., negation
of 10.2.1 implies negation of 10.2.3. Let $z\notin\partial K$. Thus
$z\in\int(K)$. Let $U\subset\int K$ be an open connected neighborhood
of $z$, e.g., the open connected component of $\int K$ that contains
$z$. Suppose there exists $f\in A(K)$ such that $\left\Vert f\right\Vert _{K}\leq1$,
$f\left(z\right)>\frac{3}{4}$ and $\abs{f\left(w\right)}<\frac{1}{4}$
for all $w\in K\backslash U$. By the maximum modulus theorem for
$U$: on the one hand, $\max_{w\in\bar{U}}\abs{f\left(w\right)}\geq f\left(z\right)>\frac{3}{4}$
while on the other $\abs{f\left(w\right)}\leq\frac{1}{4}$ for $w\in\partial U$.
Thus, $z$ cannot satisfy the Bishop $\frac{1}{4}-\frac{3}{4}$ property.
\end{proof}

\section{The set of strong boundary points for $A\left(K\right)$ is equal
to $\partial K$}
\begin{defn}
We say that $z\in K$ is a strong boundary point for $A\left(K\right)$
iff for every open neighborhood $U_{z}$, there exists $f\in A\left(K\right)$
such that $\left\Vert f\right\Vert =f\left(z\right)=1$ and for every
$w\in K-U_{z}$, $\abs{f\left(w\right)}<1$. The set of strong boundary
points for $A\left(K\right)$ is denoted $SB\left(A\left(K\right)\right)$.\end{defn}
\begin{thm}
$SB\left(A\left(K\right)\right)=\partial K$.\end{thm}
\begin{proof}
If $\int\left(K\right)=\emptyset$, then $K=\left(\partial K\right)_{\mathrm{II}}=\partial K=\mathcal{P}\left(A\left(K\right)\right)$
(see (ii) under Facts in section 1). Also, $A\left(K\right)=C(K)$
by definition of $A\left(K\right)$. Thus, $SB\left(A\left(K\right)\right)\subset K=\partial K$.
Conversely, if $z\in\partial K=K$, then by the properties of $C\left(K\right)=A\left(K\right)$,
it is easy to see that $z$ satisfies the conditions in the definition
of a strong boundary point, hence $z\in SB\left(A\left(K\right)\right)$.

If $\int\left(K\right)\neq\emptyset$, then we proceed as follows.
Let $z\in\partial K$. By theorem 7.8, $z$ is a peak point for $A\left(K\right)$.
Clearly, by definition of a peak point, $z$ satisfies the conditions
for being a strong boundary point, so $z\in SB\left(A\left(K\right)\right)$.

Conversely, let $z\in SB\left(A\left(K\right)\right)$. We must show
that $z\in\partial K$. Suppose not; then $z\in\int\left(K\right)$.
From \cite{Bachar1}, Theorem 4.1, (iv), $\int K$ is the disjoint
union of at most countably many open, connected, simply-connected
sets (the ``components'' of $\int\left(K\right)$). Let $U_{z}$
be the component that contains $z$. Select a sufficiently small closed
disk such that $z\in\bar{D}\left(z;r\right)\subset U_{z}$. With respect
to the open neighborhood, $D\left(z;r\right)$, every $f\in A\left(K\right)$
satisfying $\left\Vert f\right\Vert =f\left(z\right)=1$ can \uline{not}
satisfy ``$\abs{f\left(w\right)}<1$ for every $w\in K-D\left(z;r\right)$''
because of the maximum modulus principle. Hence, $z\in\partial K$.\end{proof}

\end{document}